\documentclass[12pt]{amsart}
\usepackage{amsmath}

\allowdisplaybreaks

\usepackage{amsfonts}

\usepackage{amsthm,amssymb}

\renewcommand{\baselinestretch}{\baselinestretch}
\renewcommand{\baselinestretch}{1.2}

\numberwithin{equation}{section}

\newtheorem{theorem}{Theorem}[section]
\newtheorem{lemma}[theorem]{Lemma}
\newtheorem{proposition}[theorem]{Proposition}
\newtheorem{corollary}[theorem]{Corollary}
\theoremstyle{definition}

\newtheorem{example}{Example}[section]

\usepackage[all,cmtip]{xy}
\usepackage{enumitem}
\usepackage{tikz-cd}
\usepackage{hyperref}
\usepackage{mathrsfs}
\usepackage{tikz}
\usepackage[english]{babel}
\usepackage{graphicx}
\usepackage{centernot}
\usepackage{mathtools}
\usepackage{cite}
\newlist{inlineroman}{enumerate*}{1}
\setlist[inlineroman]{itemjoin*={{, and }},afterlabel=~,label=\roman*.}

 \numberwithin{dummy}{section}

\newenvironment{red}{\relax\color{red}}{\relax}
\newenvironment{blue}{\relax\color{blue}}{\hspace*{.5ex}\relax}
\newcommand{\ber}{\begin{red}}
\newcommand{\er}{\end{red}}
\newcommand{\beb}{\begin{blue}}
\newcommand{\eb}{\end{blue}}

\begin{document}
	
\newcommand{\pr}{\partial}
\newcommand{\nl}{\vskip 1pc}
\newcommand{\co}{\mbox{co}}
\newcommand{\ol}{\overline}
\newcommand{\om}{\Omega}
\newcommand{\ra}{\rightarrow}
\newcommand{\epsil}{\varepsilon}
\makeatletter
\newcommand*{\rom}[1]{\expandafter\@slowromancap\romannumeral #1@}
\newcommand{\powerset}{\raisebox{.15\baselineskip}{\Large\ensuremath{\wp}}}
\newcounter{counter}       
\newcommand{\upperRomannumeral}[1]{\setcounter{counter}{#1}\Roman{counter}}
\newcommand{\lowerromannumeral}[1]{\setcounter{counter}{#1}\roman{counter}}
\newcommand*\circled[1]{\tikz[baseline=(char.base)]{
		\node[shape=circle,draw,inner sep=1pt] (char) {#1};}}
\makeatother
\title{Vector-Valued Period Polynomials and Zeta Values of Quadratic Fields}

\author{Yeong-Wook Kwon}

\author{Subong Lim}

\author{Wissam Raji}

\address{Institute of Basic Science, Korea University, 145 Anam-ro, Seongbuk-gu, Seoul, 02841, Republic of Korea}
\email{ywkwon196884@korea.ac.kr}

\address{Department of Mathematics Education, Sungkyunkwan University, Jongno-gu, Seoul 110-745, Republic of Korea}
\email{subong@skku.edu}

\address{Department of Mathematics, American University of Beirut (AUB) and the Number Theory Research Unit at the Center for Advanced Mathematical Sciences (CAMS) at AUB, Beirut, Lebanon}
\email{wr07@aub.edu.lb}

  \thanks{}
  
\begin{abstract}
Let $k\ge 2$ and $N\ge 1$ be integers. Let $D$ be a positive integer that is congruent to a square modulo $4N$, and fix $\rho$ with $\rho^2\equiv D\pmod{4N}$.
In this paper, we consider two weight $2k$ cusp forms $f^{\pm}_{k,N,D,\rho}$ on $\Gamma_0(N)$ defined by sums over binary quadratic forms, and investigate the vector-valued period polynomial arising from these forms.
Our first main result gives a closed formula for this vector-valued period polynomial. The identity component of this formula is particularly explicit: it separates as the sum of a finite \textit{algebraic part} coming from some binary forms and a \textit{zeta part} involving the values at $s=k$ of certain zeta functions. 
Using this formula together with a symmetry of vector-valued period polynomials, we explicitly evaluate, for odd $k$, the difference between the zeta values corresponding to the two choices of square root of $D$ modulo $4N$, in terms of Bernoulli numbers and a finite quadratic-form sum.
Finally, under a vanishing condition on Fricke-invariant cusp forms at lower levels, we obtain a finite divisor-sum formula for the Dedekind zeta values $\zeta_{\mathbb{Q}(\sqrt{D})}(k)$ at even integers $k$.
\end{abstract}

\maketitle
%\thispagestyle{empty}
%\tableofcontents
% \clearpage
\section{Introduction}\label{sec-intro}

The arithmetic of special values of zeta functions and $L$-functions lies at the heart of modern number theory, connecting modular forms, binary quadratic forms, and algebraic invariants of number fields.  Since Hecke’s foundational work on modular forms and zeta functions of quadratic fields \cite{Hecke}, it has been clear that modular objects often provide a natural mechanism for producing explicit expressions for zeta values and for revealing their symmetries.  A particularly effective tool in this direction is the theory of \emph{period polynomials}, which encodes the critical $L$-values of cusp forms and reflects functional equations in a clear way \cite{ManinPeriods}.  When level is involved, it is natural to work with \emph{vector-valued} period polynomials, which keep track of cosets in $\Gamma_0(N)\backslash SL_2(\mathbb{Z})$ and allow one to separate arithmetic contributions that are invisible in the scalar setting.

In parallel, binary quadratic forms provide a classical and remarkably concrete approach to quadratic fields and their zeta functions.  Zagier’s work \cite{ZagierZetaQuadratic} illustrates how sums over quadratic forms of fixed discriminant give rise to modular forms and how such constructions can be used to obtain explicit formulas for arithmetic quantities.  In the presence of congruence conditions, one is led to refined families of quadratic forms adapted to $\Gamma_0(N)$, and the associated orbit decompositions naturally produce Dirichlet-type zeta functions attached to $\Gamma_0(N)$-equivalence classes.  Concretely, these are partial zeta-type sums obtained by summing suitable powers of $Q(x,y)$ over lattice points (or equivalently over representatives in a $\Gamma_0(N)$-class), and they encode the arithmetic of the corresponding orbit in a way compatible with the level structure. On the modular side, level $N$ also introduces Fricke/Atkin--Lehner involutions \cite{AtkinLehner}, which interact with parity and with the decomposition of period polynomials into even and odd parts. 

The purpose of this paper is to make these connections explicit by studying cusp forms built from $\Gamma_0(N)$-orbits of binary quadratic forms and by computing their vector-valued period polynomials in closed form.
Precise definitions of the objects and notation appearing below are collected in Section~\ref{sec-mainresults}.

Concretely, for integers $k\ge 2$, $N\ge 1$, and for $D>0$ together with $\rho$ satisfying $\rho^2\equiv D\pmod{4N}$, we consider explicit cusp forms in $S_{2k}(N)$ defined by weighted sums over quadratic forms in $\mathcal{Q}_{N,D,\rho}$, where
\[
\mathcal{Q}_{N,D,\rho}:=\{[Na,b,c]=Nax^{2}+bxy+cy^{2}:a,b,c\in\mathbb{Z},\ b^{2}-4Nac=D,\ b\equiv\rho\pmod{2N}\}.
\]
These forms admit natural combinations $f_{k,N,D,\rho}^{\pm}$ adapted to the involution $b\mapsto -b$, and are well-suited to the Fricke operator $W_N$, enabling a clean period-polynomial analysis.

Our first main result (Theorem~\ref{thm-period}) gives an explicit expression for the vector-valued period polynomial
\[
r^+_{f^+_{k,N,D,\rho}} \;+\; r^-_{f^-_{k,N,D,\rho}}.
\]
Here, for a cusp form $f\in S_{2k}(N)$, $r_f$ denotes its vector-valued period polynomial and $r_f^{\pm}$ its even/odd parts (see Section~\ref{sec-mainresults}).
The identity component of our formula is particularly explicit: it separates an \emph{algebraic part} coming from finite sums over the ranges $a>0>c$ and $a<0<c$ of quadratic forms, and a \emph{zeta part} involving the natural zeta functions $\zeta_{N,D,\rho}(s)$ built from $\Gamma_0(N)$-equivalence classes (see Section \ref{sec-mainresults}).
Conceptually, this can be viewed as a level-$N$ period--zeta connection: period-polynomial coefficients are expressed in terms of quadratic-form sums together with explicitly normalized zeta values.

As a first application, when $k$ is odd, we obtain a closed formula for the asymmetry
\[
\zeta_{N,D,\rho}(k)\;-\;\zeta_{N,D,-\rho}(k)
\]
(Theorem~\ref{thm-zetaformula}). This result isolates the precise contribution of the congruence condition $\rho$ and shows how it is detected by parity considerations in the period-polynomial setting. In particular, it exhibits how Fricke symmetry and the even/odd decomposition of periods translate into a nontrivial arithmetic identity between zeta values.

Finally, under a natural vanishing condition on Fricke-invariant cusp forms at the relevant lower levels, we apply the preceding period computations to deduce an explicit formula for the Dedekind zeta value $\zeta_{\mathbb{Q}(\sqrt{D})}(k)$ at even integers $k$ (Theorem~\ref{thm-dedekindzeta}), expressed as a finite divisor-sum with congruence restrictions. This provides computable expressions for $\zeta_{\mathbb{Q}(\sqrt{D})}(k)$ that refine the classical Hecke-type evaluations and illustrate the strength of period methods in the presence of level structure.

\section{Main Results}\label{sec-mainresults}

Let $k\geq 2$ be a positive integer and let $N$ be a positive integer. 
We write $i:=\sqrt{-1}$ and denote by $I$ the $2\times 2$ identity matrix. Additionally, we denote by $\mathbb{H}$ the complex upper half-plane.
Let
$$\Gamma_{0}(N):=\left\{\left(\begin{smallmatrix}
    \alpha & \beta \\ \gamma & \delta
\end{smallmatrix}\right)\in\mathrm{SL}_{2}(\mathbb{Z}):\gamma\equiv 0\pmod{N}\right\}.$$
Furthermore, let $D$ be a positive integer which is congruent to a square modulo $4N$ and let $\rho$ be an integer with $\rho^{2}\equiv D\pmod{4N}$. Consider the set
$$\mathcal{Q}_{N,D,\rho}:=\{[Na,b,c]=Nax^{2}+bxy+cy^{2}:a,b,c\in\mathbb{Z},~b^{2}-4Nac=D,~b\equiv\rho\pmod{2N}\}$$
on which the congruence subgroup $\Gamma_{0}(N)$ acts by
$$\left(Q\circ\left(\begin{smallmatrix}
    \alpha & \beta \\ \gamma & \delta
\end{smallmatrix}\right)\right)(x,y)=Q(\alpha x+\beta y,\gamma x+\delta y)\quad\left(Q\in\mathcal{Q}_{N,D,\rho},~\left(\begin{smallmatrix}
    \alpha & \beta \\ \gamma & \delta
\end{smallmatrix}\right)\in\Gamma_{0}(N)\right).$$

Define
\begin{align*}
    f_{k,N,D,\rho}(z)&:=\frac{D^{k-1/2}}{\pi\binom{2k-2}{k-1}}\cdot\frac{1}{2}\sum_{[Na,b,c]\in\mathcal{Q}_{N,D,\rho}}(Naz^{2}+bz+c)^{-k},\\
    f_{k,N,D,\rho}'(z)&:=\frac{D^{k-1/2}}{\pi\binom{2k-2}{k-1}}\cdot\frac{1}{2}\sum_{[Na,b,c]\in\mathcal{Q}_{N,D,\rho}'}(Naz^{2}+bz+c)^{-k},\\
    f_{k,N,D,\rho}^{+}&:=f_{k,N,D,\rho}+f_{k,N,D,\rho}',\quad f_{k,N,D,\rho}^{-}:=i(f_{k,N,D,\rho}-f_{k,N,D,\rho}'),
\end{align*}
where
$$\mathcal{Q}_{N,D,\rho}':=\{[Na,-b,c]:[Na,b,c]\in\mathcal{Q}_{N,D,\rho}\}.$$
Then $f_{k,N,D,\rho}, f_{k,N,D,\rho}', f_{k,N,D,\rho}^{\pm}\in S_{2k}(N)$, where $S_{2k}(N)$ denotes the space of cusp forms of weight $2k$ on $\Gamma_{0}(N)$.

For each $\mathcal{C}\in\mathcal{Q}_{N,D,\rho}/\Gamma_{0}(N)$, fix any representative $Q_{\mathcal{C}}\in\mathcal{C}$ and write
$$\Gamma_{0}(N)_{Q_{\mathcal{C}}}:=\{A\in\Gamma_{0}(N):Q_{\mathcal{C}}\circ A=Q_{\mathcal{C}}\}.$$
For $\mathrm{Re}(s)>1$, we define
$$\zeta_{\mathcal{C}}^{(N)}(s):=\sum\limits_{\substack{(b,d)\in\mathbb{Z}^{2}/\Gamma_{0}(N)_{Q_{\mathcal{C}}} \\ Q_{\mathcal{C}}(b,d)>0 \\ \gcd(d,N)=1}}Q_{\mathcal{C}}(b,d)^{-s},$$
and set
$$\zeta_{N,D,\rho}(s):=\sum_{\mathcal{C}\in\mathcal{Q}_{N,D,\rho}/\Gamma_{0}(N)}\zeta_{\mathcal{C}}^{(N)}(s).$$

For $A=\left(\begin{smallmatrix}
    \alpha & \beta \\ \gamma & \delta
\end{smallmatrix}\right)\in\mathrm{GL}_{2}^{+}(\mathbb{R})$ and a holomorphic function $f:\mathbb{H}\rightarrow\mathbb{C}$, put
$$(f|_{2k}A)(z):=(\det A)^{k}(\gamma z+\delta)^{-2k}f\left(\frac{\alpha z+\beta}{\gamma z+\delta}\right).$$
We set $V_{2k-2}$ to be the $\mathbb{C}$-vector space of polynomials in $X$ of degree at most $2k-2$. For simplicity, let $e_{N}:=[\mathrm{SL}_{2}(\mathbb{Z}):\Gamma_{0}(N)]$.
For $f\in S_{2k}(N)$, define the \textit{vector-valued period polynomial} $r_{f}:\mathbb{C}\rightarrow\mathbb{C}^{e_{N}}$ of $f$ by the vector
$$r_{f}(X):=(r_{f}(A)(X))_{A\in\Gamma_{0}(N)\backslash\mathrm{SL}_{2}(\mathbb{Z})}$$
where
$$r_{f}(A)(X):=\int_{0}^{i\infty}(f|_{2k}A)(z)(X-z)^{2k-2}dz\in V_{2k-2}.$$
We write $r_{f}^{+}(A)(X)$ (resp. $r_{f}^{-}(A)(X)$) for the even (resp.\ odd) part of $r_{f}(A)(X)$ i.e.
$$r_{f}^{\pm}(A)(X):=\frac{1}{2}(r_{f}(A)(X)\pm r_{f}(A)(-X)).$$

Our first main result gives an expression for each component of the vector-valued period polynomial $r_{f_{k,N,D,\rho}^{+}}^{+}+r_{f_{k,N,D,\rho}^{-}}^{-}$.

\begin{theorem}\label{thm-period}
Let $k,N,D$ and $\rho$ be as above. Assume that $D$ is not a perfect square. Set
$$c_{k,D}:=\binom{2k-2}{k-1}D^{1/2-k}\pi.$$
For $A\in\Gamma_{0}(N)\backslash \mathrm{SL}_{2}(\mathbb{Z})$, the $A$-th component of the vector-valued period polynomial $r_{f_{k,N,D,\rho}^{+}}^{+}+r_{f_{k,N,D,\rho}^{-}}^{-}$ is given by
$$r_{f_{k,N,D,\rho}^{+}}^{+}(A)(X)+r_{f_{k,N,D,\rho}^{-}}^{-}(A)(X)=\sum_{n=0}^{2k-2}i^{-n+1}\binom{2k-2}{n}r_{n}(A)X^{2k-2-n},$$
where
\begin{equation*}
    r_{n}(A)=\frac{i^{n^{2}}}{2c_{k,D}}\lim_{\varepsilon\rightarrow 0^{+}}\int_{\varepsilon}^{1/\varepsilon}\left(\sum_{[a,b,c]\in\mathcal{Q}_{N,D,\rho}\circ A}+(-1)^{n}\sum_{[a,b,c]\in\mathcal{Q}_{N,D,\rho}'\circ A}\frac{t^{n}}{(-at^{2}+bit+c)^{k}}\right)dt.
\end{equation*}
In particular,
\begin{align*}
r_{f_{k,N,D,\rho}^{+}}^{+}(I)(X)&+r_{f_{k,N,D,\rho}^{-}}^{-}(I)(X)\\
&=\sum_{\substack{[Na,b,c]\in\mathcal{Q}_{N,D,\rho}\\ a>0>c}}
(NaX^{2}-bX+c)^{k-1}-\sum_{\substack{[Na,b,c]\in\mathcal{Q}_{N,D,\rho}\\ a<0<c}}(NaX^{2}-bX+c)^{k-1}\\
&\quad-\frac{\pi }{N(2k-1)c_{k,D}}
\cdot
\frac{\zeta_{N,D,\rho}(k)+(-1)^{k}\zeta_{N,D,-\rho}(k)}{\zeta(2k)\prod_{p\mid N}(1-p^{-2k})}X^{2k-2}\\
&\quad+\frac{\pi }{N^{k}(2k-1)c_{k,D}}
\cdot
\frac{\zeta_{N,D,-\rho}(k)+(-1)^{k}\zeta_{N,D,\rho}(k)}{\zeta(2k)\prod_{p\mid N}(1-p^{-2k})}.
\end{align*}
\end{theorem}

As an application of Theorem \ref{thm-period}, we obtain a closed formula for the difference $\zeta_{N,D,\rho}(k)-\zeta_{N,D,-\rho}(k)$ when $k$ is odd.

\begin{theorem}\label{thm-zetaformula}
    Let $k,N,D$ and $\rho$ be as in Theorem \ref{thm-period}. If $k$ is odd, then
    \begin{align*}
        \zeta_{N,D,\rho}(k)-\zeta_{N,D,-\rho}(k)&=\frac{2^{2k-1}N^{k}(2k-1)D^{1/2-k}B_{2k}\pi^{2k}}{(2k)!}\binom{2k-2}{k-1}\\
        &\qquad\qquad\times\prod_{p\mid N}(1-p^{-2k})\left(\sum\limits_{\substack{[Na,b,c]\in\mathcal{Q}_{N,D,\rho} \\ a>0>c}}c^{k-1}-\sum\limits_{\substack{[Na,b,c]\in\mathcal{Q}_{N,D,\rho} \\ a<0<c}}c^{k-1}\right),
    \end{align*}
    where $B_{n}$ denotes the $n$th Bernoulli number.
\end{theorem}

For a positive integer $M$, let $W_{M}$ denote the Fricke involution:
$$W_{M}:=\left(\begin{smallmatrix}
    0 & -1/\sqrt{M} \\ \sqrt{M} & 0
\end{smallmatrix}\right).$$
Next, given a positive divisor $d$ of $N$, define
$$S_{2k}^{+}(N/d):=\{f\in S_{2k}(N/d):f|_{2k}W_{N/d}=f\}.$$
Using Theorem \ref{thm-period}, we obtain the following formula for the value of the Dedekind zeta function $\zeta_{\mathbb{Q}(\sqrt{D})}(s)$ at $s=k$.

\begin{theorem}\label{thm-dedekindzeta}
    Let $k$ be an even positive integer, let $N$ be a positive integer, and let $D>0$ be a fundamental discriminant such that $D\equiv 1\pmod{4N}$. Assume that $S_{2k}^{+}(N/d)=\{0\}$ for every positive divisor $d$ of $N$.
    Then
    \begin{align*}
        \zeta_{\mathbb{Q}(\sqrt{D})}(k)&=(2k-1)D^{1/2-k}\zeta(2k)N^{k}\binom{2k-2}{k-1}\\
        &\quad\times\sum_{d\mid N}d^{-2k}\prod_{p\mid(N/d)}(1-p^{-2k})\sum\limits_{\substack{|b|<\sqrt{D} \\ b\equiv 1\pmod{2N/d}}}\sigma_{k-1}\left(\frac{d(D-b^{2})}{4N}\right).
    \end{align*}
\end{theorem}

In the following example, we utilize Theorem \ref{thm-dedekindzeta} to evaluate $\zeta_{\mathbb{Q}(\sqrt{17})}(2)$ and $\zeta_{\mathbb{Q}(\sqrt{145})}(2)$.

\begin{example}
    \begin{enumerate}
        \item[(i)] Let $N=2$ and $D=17$. Then $D\equiv 1\pmod{4N}$. By \cite[Lemma 2.2]{CK13},
        $$\dim S_{4}^{+}(2)=0,$$
        and, by \cite[Corollary 8.8]{K92},
        $$\dim S_{4}^{+}(2/2)=\dim S_{4}^{+}(1)=\dim S_{4}(\mathrm{SL}_{2}(\mathbb{Z}))=0.$$
        Applying Theorem \ref{thm-dedekindzeta}, we obtain
        \begin{align*}
            \zeta_{\mathbb{Q}(\sqrt{17})}(2)=&24\times  17^{-3/2}\times\zeta(4)\\
            &\times\left(\frac{15}{16}\sum\limits_{\substack{|b|<\sqrt{17} \\ b\equiv 1\pmod{4}}}\sigma_{1}\left(\frac{17-b^{2}}{8}\right)+\frac{1}{16}\sum\limits_{\substack{|b|<\sqrt{17} \\ b\equiv 1\pmod{2}}}\sigma_{1}\left(\frac{17-b^{2}}{4}\right)\right).
        \end{align*}
        Note that
        \begin{align*}
            \{b\in\mathbb{Z}:|b|<\sqrt{17},~b\equiv 1\pmod{2}\}&=\{-3,-1,1,3\},\\
            \{b\in\mathbb{Z}:|b|<\sqrt{17},~b\equiv 1\pmod{4}\}&=\{1,-3\}.
        \end{align*}
        Thus,
        \begin{align*}
            \sum\limits_{\substack{|b|<\sqrt{17} \\ b\equiv 1\pmod{4}}}\sigma_{1}\left(\frac{17-b^{2}}{8}\right)&=\sigma_{1}(2)+\sigma_{1}(1)=4,\\
            \sum\limits_{\substack{|b|<\sqrt{17} \\ b\equiv 1\pmod{2}}}\sigma_{1}\left(\frac{17-b^{2}}{4}\right)&=2\sigma_{1}(4)+2\sigma_{1}(2)=20.
        \end{align*}
        Meanwhile, it is known that $\zeta(4)=\pi^{4}/90$. Consequently,
        $$\zeta_{\mathbb{Q}(\sqrt{17})}(2)=24\times 17^{-3/2}\times\frac{\pi^{4}}{90}\times\left(\frac{15}{16}\times 4+\frac{1}{16}\times 20\right)=\frac{4\pi^{4}}{51\sqrt{17}}.$$

        \item[(ii)] Let $N=3$ and $D=145$. Then $4N=12$ and $D\equiv 1\pmod{12}$. Using \cite[Lemma 2.2]{CK13} and \cite[Corollary 8.8]{K92}, one can check that
        $$\dim S_{4}^{+}(3)=\dim S_{4}^{+}(3/3)=0.$$
        Hence, by Theorem \ref{thm-dedekindzeta}, we have
        \begin{align*}
            \zeta_{\mathbb{Q}(\sqrt{145})}(2)=&54\times 145^{-3/2}\times\zeta(4)\\
            &\times\left(\frac{80}{81}\sum\limits_{\substack{|b|<\sqrt{145}\\ b\equiv 1\pmod{6}}}\sigma_{1}\left(\frac{145-b^{2}}{12}\right)+\frac{1}{81}\sum\limits_{\substack{|b|<\sqrt{145}\\ b\equiv 1\pmod{2}}}\sigma_{1}\left(\frac{145-b^{2}}{4}\right)\right).
        \end{align*}
        A similar computation to part (i) shows that
        \begin{equation*}
            \sum\limits_{\substack{|b|<\sqrt{145}\\ b\equiv 1\pmod{6}}}
            \sigma_{1}\left(\frac{145-b^{2}}{12}\right)=64\quad\text{and}\quad\sum\limits_{\substack{|b|<\sqrt{145}\\ b\equiv 1\pmod{2}}}\sigma_{1}\left(\frac{145-b^{2}}{4}\right)=640.
        \end{equation*}
        Therefore,
        $$\zeta_{\mathbb{Q}(\sqrt{145})}(2)=54\times 145^{-3/2}\times\frac{\pi^{4}}{90}\times\left(\frac{80}{81}\times 64+\frac{1}{81}\times 640\right)=\frac{128\pi^{4}}{435\sqrt{145}}.$$ 
    \end{enumerate}
\end{example}

The remainder of this paper is organized as follows. In Section \ref{sec-prel}, we review vector-valued period polynomials and binary quadratic forms. 
In Section \ref{sec-period}, we prove Theorem \ref{thm-period}. Next, in Section \ref{sec-zetaformula}, we establish Theorem \ref{thm-zetaformula}. Finally, in Section \ref{sec-dedekindzeta}, we provide a proof of Theorem \ref{thm-dedekindzeta}.

\section{Preliminaries}\label{sec-prel}

Let $k\geq 2$ be a positive integer, and let $V_{2k-2}$ be the space of polynomials of degree at most $2k-2$ with coefficients in $\mathbb{C}$.
Then the map
\begin{align*}
    V_{2k-2}\times\mathrm{GL}_{2}(\mathbb{R})&\rightarrow V_{2k-2}\\
    \left(P,g=\left(\begin{smallmatrix}
        a & b \\ c & d
    \end{smallmatrix}\right)\right)&\mapsto (P|_{2-2k}g)(z):=(cz+d)^{2k-2}P\left(\frac{az+b}{cz+d}\right)
\end{align*}
is a right group action of $\mathrm{GL}_{2}(\mathbb{R})$ on $V_{2k-2}$.
Let $N$ be a positive integer, and denote by $\widetilde{V}_{2k-2}^{\Gamma_{0}(N)}$ the space of maps $P:\Gamma_{0}(N)\backslash\mathrm{SL}_{2}(\mathbb{Z})\rightarrow V_{2k-2}$. 
On this space, we define an $\mathrm{SL}_{2}(\mathbb{Z})$-action as follows: for $P\in\widetilde{V}_{2k-2}^{\Gamma_{0}(N)}$ and $g\in\mathrm{SL}_{2}(\mathbb{Z})$,
$$(P|g)(A):=P(Ag^{-1})|_{2-2k}g.$$
Let $S=\left(\begin{smallmatrix}
    0 & -1 \\ 1 & 0
\end{smallmatrix}\right)$, $T=\left(\begin{smallmatrix}
    1 & 1 \\ 0 & 1
\end{smallmatrix}\right)$, and $U=TS$. 
Then the set
$$W_{2k-2}^{\Gamma_{0}(N)}:=\{P\in \widetilde{V}_{2k-2}^{\Gamma_{0}(N)}:P|(-I)=P,~P+P|S=P+P|U+P|U^{2}=0\}$$
is a subspace of $\widetilde{V}_{2k-2}^{\Gamma_{0}(N)}$, and the elements of this subspace are called \textit{vector-valued period polynomials}.

Now denote by $S_{2k}(N)$ the space of cusp forms of weight $2k$ on $\Gamma_{0}(N)$. 
Given $f\in S_{2k}(N)$ and $A\in\mathrm{SL}_{2}(\mathbb{Z})$, define
\begin{equation}\label{eq-vvpp}
    r_{f}(A)(X):=\int_{0}^{i\infty}(f|_{2k}A)(z)(X-z)^{2k-2}dz.
\end{equation}
Then
\begin{equation}\label{eq-vvppexp}
    r_{f}(A)(X)=\sum_{n=0}^{2k-2}i^{-n+1}\binom{2k-2}{n}r_{n,f}(A)X^{2k-2-n},
\end{equation}
where
$$r_{n,f}(A):=\int_{0}^{\infty}(f|_{2k}A)(it)t^{n}dt.$$
In particular, $r_{f}(A)\in V_{2k-2}$. 
Since $f|_{2k}\gamma=f$ for all $\gamma\in\Gamma_{0}(N)$, $r_{f}(A)$ depends only on the right coset of $A$ in $\Gamma_{0}(N)\backslash\mathrm{SL}_{2}(\mathbb{Z})$. Hence, the map
\begin{align*}
    r_{f}:\Gamma_{0}(N)\backslash\mathrm{SL}_{2}(\mathbb{Z})&\rightarrow V_{2k-2}\\
    A&\mapsto r_{f}(A)
\end{align*}
is a well-defined map belonging to the space $\widetilde{V}_{2k-2}^{\Gamma_{0}(N)}$. 
In fact, it is known that $r_{f}\in W_{2k-2}^{\Gamma_{0}(N)}$ (for more details, see \cite[p. 716]{PP13}). This map can be identified with the vector-valued period polynomial
$$\mathbb{C}\rightarrow\mathbb{C}^{e_{N}},\quad X\mapsto (r_{f}(A)(X))_{A\in\Gamma_{0}(N)\backslash\mathrm{SL}_{2}(\mathbb{Z})}.$$
Here, $e_{N}=[\mathrm{SL}_{2}(\mathbb{Z}):\Gamma_{0}(N)]$. Abusing notation, we also denote this polynomial by $r_{f}$ and call it the \textit{vector-valued period polynomial of $f$}. We write $r_{f}^{+}(A)(X)$ (resp. $r_{f}^{-}(A)(X)$) for the \textit{even} (resp. \textit{odd}) \textit{part} of $r_{f}(A)(X)$ i.e.
$$r_{f}^{\pm}(A)(X):=\frac{1}{2}(r_{f}(A)(X)\pm r_{f}(A)(-X)).$$

Let $D$ be a positive integer such that $D\equiv 0$ or $1\pmod{4}$, and set
\begin{align*}
    \mathcal{Q}_{D}&:=\{[a,b,c]=ax^{2}+bxy+cy^{2}:a,b,c\in\mathbb{Z},~b^{2}-4ac=D\},\\
    \mathcal{Q}_{D}^{0}&:=\{[a,b,c]\in\mathcal{Q}_{D}:\mathrm{gcd}(a,b,c)=1\}.
\end{align*}
Then $\mathrm{SL}_{2}(\mathbb{Z})$ acts on $\mathcal{Q}_{D}$ by
$$\left(Q\circ\left(\begin{smallmatrix}
    \alpha & \beta \\ \gamma & \delta
\end{smallmatrix}\right)\right)(x,y)=Q(\alpha x+\beta y,\gamma x+\delta y),$$
and $\mathcal{Q}_{D}^{0}$ is a $\mathrm{SL}_{2}(\mathbb{Z})$-invariant subset of $\mathcal{Q}_{D}$.

For a positive integer $D$ which is congruent to a square modulo $4N$ and an integer $\rho$ with $\rho^{2}\equiv D\pmod{4N}$, we set
$$\mathcal{Q}_{N,D,\rho}:=\{[Na,b,c]=Nax^{2}+bxy+cy^{2}:a,b,c\in\mathbb{Z},~b^{2}-4Nac=D,~b\equiv\rho\pmod{2N}\}.$$
Then the action of $\mathrm{SL}_{2}(\mathbb{Z})$ on $\mathcal{Q}_{D}$ restricts to an action of $\Gamma_{0}(N)$ on $\mathcal{Q}_{N,D,\rho}$.
If we set
$$\mathcal{Q}_{N,D,\rho}^{0}:=\{[Na,b,c]\in\mathcal{Q}_{N,D,\rho}:\mathrm{gcd}(Na,b,c)=1\},$$
then it is a $\Gamma_{0}(N)$-invariant subset of $\mathcal{Q}_{N,D,\rho}$ and
$$\mathcal{Q}_{N,D,\rho}=\bigcup_{\ell^{2}\mid D}\bigcup\limits_{\substack{\lambda\pmod{2N} \\ \lambda^{2}\equiv D/\ell^{2}\pmod{4N} \\ \ell\lambda\equiv\rho\pmod{2N}}}\ell\cdot\mathcal{Q}_{N,D/\ell^{2},\lambda}^{0}.$$
Thus, we can reduce to the study of forms $Q\in\mathcal{Q}_{N,D,\rho}^{0}$, which we call $\Gamma_{0}(N)$-\textit{primitive}.
Set
\begin{equation}\label{eq-gcd}
    m:=\gcd\left(N,\rho,\frac{\rho^{2}-D}{4N}\right).
\end{equation}
Then $m$ depends only on $\rho\pmod{2N}$ since replacing $\rho$ by $\rho+2N$ replaces $\frac{\rho^{2}-D}{4N}$ by $\frac{\rho^{2}-D}{4N}+\rho+N$. 
For $[Na,b,c]\in\mathcal{Q}_{N,D,\rho}^{0}$, we have $\gcd(N,b,ac)=m$ and $\gcd(a,b,c)=1$, so the two numbers
\begin{equation}\label{eq-mdecomp}
    \gcd(N,b,c)=m_{1}\quad\text{and}\quad\gcd(N,a,b)=m_{2}
\end{equation}
are coprime and have product $m$. 
Conversely, we have

\begin{proposition}\label{prop-gkz}
    Define $m$ by \eqref{eq-gcd} and fix a decomposition $m=m_{1}m_{2}$ with $m_{1},m_{2}>0$ and $\gcd(m_{1},m_{2})=1$. Then there is a one-to-one correspondence between the $\Gamma_{0}(N)$-equivalence classes of forms $Q\in\mathcal{Q}_{N,D,\rho}^{0}$ satisfying \eqref{eq-mdecomp} and the $\mathrm{SL}_{2}(\mathbb{Z})$-equivalence classes of forms in $\mathcal{Q}_{D}^{0}$ given by
    $$Q=[Na,b,c]\mapsto\widetilde{Q}=[N_{1}a,b,N_{2}c];$$
    here $N_{1}\cdot N_{2}$ is any decomposition of $N$ into coprime positive factors satisfying $\gcd(m_{1},N_{2})=\gcd(m_{2},N_{1})=1$. In particular, $|\mathcal{Q}_{N,D,\rho}^{0}/\Gamma_{0}(N)|=2^{\nu}\cdot |\mathcal{Q}_{D}^{0}/\mathrm{SL}_{2}(\mathbb{Z})|$, where $\nu$ is the number of prime factors of $m$.
\end{proposition}

\begin{proof}
    See \cite[pp. 505--506]{GKZ87}.
\end{proof}

Denote by $\mathcal{Q}_{N,D}$ (resp. $\mathcal{Q}_{N,D}^{0}$) the set of all (resp. all $\Gamma_{0}(N)$-primitive) quadratic forms $[Na,b,c]$ of discriminant $D$.
The Fricke involution $W_{N}=\left(\begin{smallmatrix}
    0 & -1/\sqrt{N} \\ \sqrt{N} & 0
\end{smallmatrix}\right)$ acts on $\mathcal{Q}_{N,D}$ as follows: For $Q=[Na,b,c]\in\mathcal{Q}_{N,D}$,
\begin{equation}\label{eq-Fricke}
    Q\circ W_{N}:=[Nc,-b,a].
\end{equation}
For an integer $\rho$ with $\rho^{2}\equiv D\pmod{4N}$, \eqref{eq-Fricke} induces a one-to-one correspondence between $\mathcal{Q}_{N,D,\rho}$ and $\mathcal{Q}_{N,D,-\rho}$.

\section{Proof of Theorem \ref{thm-period}}\label{sec-period}

In this section, we prove Theorem \ref{thm-period}. Throughout this section, $N,D$ and $\rho$ are as in Section \ref{sec-prel}. To prove Theorem \ref{thm-period}, we need the following lemma.

\begin{lemma}\label{lem-limit}
    Let $\mathcal{C}$ be a $\Gamma_{0}(N)$-equivalence class in $\mathcal{Q}_{N,D,\rho}$, and fix $Q_{\mathcal{C}}\in\mathcal{C}$.
    For $\mathrm{Re}(s)>1$, define
    \begin{equation}\label{eq-classzeta}
        \zeta_{\mathcal{C}}^{(N)}(s):=\sum\limits_{\substack{(b,d)\in\mathbb{Z}^{2}/\Gamma_{0}(N)_{Q_{\mathcal{C}}} \\ Q_{\mathcal{C}}(b,d)>0 \\ \gcd(d,N)=1}}Q_{\mathcal{C}}(b,d)^{-s}.
    \end{equation}
    Then for every integer $n$ with $0\leq n\leq 2k-2$ we have
    \begin{equation}\label{eq-limitvalue}
        \lim_{\varepsilon\rightarrow 0^{+}}\sum_{\substack{[Na,b,c]\in \mathcal{C} \\ c>0}}\int_{-\varepsilon}^{\varepsilon}\frac{t^{n}}{(-Nat^{2}+bit+c)^{k}}dt=\delta_{n,0}\frac{2\pi}{N(2k-1)}\cdot\frac{\zeta_{\mathcal{C}}^{(N)}(k)}{\zeta(2k)\prod_{p\mid N}(1-p^{-2k})}.
\end{equation}
\end{lemma}

\begin{proof}
Fix $n\in\{0,1,\ldots,2k-2\}$ and $\varepsilon>0$. Put
$$L_{\varepsilon}:=\sum\limits_{\substack{[Na,b,c]\in\mathcal{C} \\ c>0}}\int_{-\varepsilon}^{\varepsilon}\frac{t^{n}}{(-Nat^{2}+bit+c)^{k}}dt.$$
Setting $t=\varepsilon u$, we get
\begin{equation}\label{eq-substitution}
    L_{\varepsilon}=\varepsilon^{n+1}\sum\limits_{\substack{[Na,b,c]\in\mathcal{C} \\ c>0}}\int_{-1}^{1}\frac{u^{n}}{(-Na\varepsilon^{2}u^{2}+bi\varepsilon u+c)^{k}}du.
\end{equation}

Let
$$H:=\left\{\gamma_{r}:=\left(\begin{smallmatrix}
    1 & 0 \\ Nr & 1
\end{smallmatrix}\right):r\in\mathbb{Z}\right\}.$$
Then $H$ is a subgroup of $\Gamma_{0}(N)$ and acts on $\mathcal{C}$ by $Q\mapsto Q\circ\gamma_{r}$.
A direct computation shows that, for $Q=[Na,b,c]\in\mathcal{C}$,
\begin{equation*}
Q\circ\gamma_{r}=[N(a+rb+r^2Nc),b+2rNc,c].
\end{equation*}
In particular, $c$ is unchanged, and as $r$ varies, $b$ runs through the congruence classes modulo $2Nc$.
Fix, for each $c>0$, a complete set of representatives for the residue classes $b\pmod{2Nc}$.
Then, for each fixed $c>0$, the map
\begin{align*}
    \{[Na,b,c]\in\mathcal{C}:b\pmod{2Nc}\}\times\mathbb{Z}&\rightarrow\{[Na,b,c]\in\mathcal{C}:c>0\}\\
    (Q_{0},r)&\mapsto Q_{0}\circ\gamma_{r}
\end{align*}
is a bijection. 
Hence, \eqref{eq-substitution} becomes
\begin{equation}\label{eq-doublesum}
L_{\varepsilon}=\varepsilon^{n+1}\sum_{\substack{[Na,b,c]\in\mathcal{C} \\ c>0 \\ b\pmod{2Nc}}}\sum_{r\in\mathbb{Z}}\int_{-1}^{1}
\frac{u^{n}}{(-Na_{r}\varepsilon^{2}u^{2}+i(b+2rNc)\varepsilon u+c)^k}du,
\end{equation}
where $a_{r}:=a+rb+r^{2}Nc$.

Using $b^{2}-4Nac=D$, one can check that
\begin{equation*}
-Na_{r}t^{2}+i(b+2rNc)t+c=c\left(\frac{D}{4c^{2}}t^{2}+\left(1+i\left(rN+\frac{b}{2c}\right)t\right)^{2}\right).
\end{equation*}
Applying this with $t=\varepsilon u$ in \eqref{eq-doublesum} yields
\begin{align}
L_{\varepsilon}&=\varepsilon^{n+1}\sum_{\substack{[Na,b,c]\in\mathcal{C} \\ c>0 \\ b\pmod{2Nc}}}c^{-k}\int_{-1}^{1}u^{n}\sum_{r\in\mathbb Z}\frac{du}{\left(\frac{D\varepsilon^{2}u^{2}}{4c^{2}}+\left(1+i\varepsilon u(rN+\frac{b}{2c})\right)^{2}\right)^{k}}\nonumber\\
&=\varepsilon^{n}\sum_{\substack{[Na,b,c]\in\mathcal{C} \\ c>0 \\ b\pmod{2Nc}}}c^{-k}\left[\varepsilon\sum_{r\in\mathbb{Z}}\int_{-1}^{1}\frac{u^{n}du}{\left(\frac{D\varepsilon^{2}u^{2}}{4c^{2}}+\left(1+i\varepsilon u(rN+\frac{b}{2c})\right)^{2}\right)^{k}}\right].\nonumber
\end{align}
The inner sum is the Riemann sum for the integral
$$\frac{1}{N}\int_{-\infty}^{\infty}\int_{-1}^{1}\frac{u^{n}}{(1+ixu)^{2k}}dudx.$$
Therefore, $\lim_{\varepsilon\rightarrow 0^{+}}L_{\varepsilon}=0$ unless $n=0$. When $n=0$, it equals
$$\frac{2\pi}{N(2k-1)}\sum_{\substack{[Na,b,c]\in\mathcal{C} \\ c>0 \\ b\pmod{2Nc}}}c^{-k}$$
(see \cite[p. 226]{KZ84}). 
It remains to show that
$$\sum_{\substack{[Na,b,c]\in\mathcal{C} \\ c>0 \\ b\pmod{2Nc}}}c^{-k}=\frac{\zeta_{\mathcal{C}}^{(N)}(k)}{\zeta(2k)\prod_{p\mid N}(1-p^{-2k})}.$$
Writing any $Q\in\mathcal{C}$ as $Q_{\mathcal{C}}\circ A$
with $A\in\Gamma_{0}(N)$, we have
$$(Q_{\mathcal{C}}\circ A)(0,1)=Q_{\mathcal{C}}(b,d)\quad\text{for}~A=\left(\begin{smallmatrix}
    * & b \\ * & d
\end{smallmatrix}\right),$$
and therefore
\begin{align*}
    \zeta_{\mathcal{C}}^{(N)}(k)&=\sum\limits_{\substack{(b,d)\in\mathbb{Z}^{2}/\Gamma_{0}(N)_{Q_{\mathcal{C}}} \\ Q_{\mathcal{C}}(b,d)>0 \\ \gcd(d,N)=1}}Q_{\mathcal{C}}(b,d)^{-k}=\sum\limits_{\substack{\ell\geq 1 \\ \gcd(\ell,N)=1}}\ell^{-2k}\sum\limits_{\substack{(b,d)\in\mathbb{Z}^{2}/\Gamma_{0}(N)_{Q_{\mathcal{C}}} \\ Q_{\mathcal{C}}(b,d)>0 \\ \gcd(d,N)=1 \\ \gcd(b,d)=1}}Q_{\mathcal{C}}(b,d)^{-k}\\
    &=\zeta(2k)\prod_{p\mid N}(1-p^{-2k})\sum_{\substack{[Na,b,c]\in\mathcal{C} \\ c>0 \\ b\pmod{2Nc}}}c^{-k}.
\end{align*}
\end{proof}

We are now ready to prove Theorem \ref{thm-period}.

\begin{proof}[Proof of Theorem \ref{thm-period}]
Let $A=\left(\begin{smallmatrix}
    \alpha & \beta \\ \gamma & \delta
\end{smallmatrix}\right)\in\mathrm{SL}_{2}(\mathbb{Z})$.
Fix $n\in\{0,1,\ldots,2k-2\}$, and choose the sign $\pm$ so that $\pm 1=(-1)^n$. In other words, take $+$ if $n$ is even and $-$ if $n$ is odd.
Note that
\begin{align*}
    (f_{k,N,D,\rho}|_{2k}A)(z)&=\frac{1}{2}c_{k,D}^{-1}\sum_{Q\in\mathcal{Q}_{N,D,\rho}}(\gamma z+\delta)^{-2k}Q\left(\frac{\alpha z+\beta}{\gamma z+\delta},1\right)^{-k}\\
    &=\frac{1}{2}c_{k,D}^{-1}\sum_{Q\in\mathcal{Q}_{N,D,\rho}}(Q\circ A)(z,1)^{-k}\\
    &=\frac{1}{2}c_{k,D}^{-1}\sum_{Q\in\mathcal{Q}_{N,D,\rho}\circ A}Q(z,1)^{-k}.
\end{align*}
Similarly,
$$(f_{k,N,D,\rho}'|_{2k}A)(z)=\frac{1}{2}c_{k,D}^{-1}\sum_{Q\in\mathcal{Q}_{N,D,\rho}'\circ A}Q(z,1)^{-k}.$$
Thus,
\begin{equation}\label{eq-periodaslimit}
    \begin{split}
        2i^{-n^{2}}c_{k,D}&r_{n,f_{k,N,D,\rho}^{\pm}}(A)\\&=\int_{0}^{\infty}\left(\sum_{[a,b,c]\in\mathcal{Q}_{N,D,\rho}\circ A}\pm\sum_{[a,b,c]\in\mathcal{Q}_{N,D,\rho}'\circ A}\right)\frac{t^{n}}{(Na(it)^{2}+bit+c)^{k}}dt\\
        &=\lim_{\varepsilon\rightarrow 0^{+}}\int_{\varepsilon}^{1/\varepsilon}\left(\sum_{[a,b,c]\in\mathcal{Q}_{N,D,\rho}\circ A}\pm\sum_{[a,b,c]\in\mathcal{Q}_{N,D,\rho}'\circ A}\right)\frac{t^{n}}{(Na(it)^{2}+bit+c)^{k}}dt.
    \end{split}
\end{equation}
From now on, we concentrate on the case where $A=I$. In this case, \eqref{eq-periodaslimit} becomes
\begin{align*}
    2i^{-n^{2}}c_{k,D}r_{n,f_{k,N,D,\rho}^{\pm}}(I)&=\int_{0}^{\infty}\left(\sum_{[Na,b,c]\in\mathcal{Q}_{N,D,\rho}}\pm\sum_{[Na,b,c]\in\mathcal{Q}_{N,D,\rho}'}\right)\frac{t^{n}}{(Na(it)^{2}+bit+c)^{k}}dt\\
    &=\lim_{\varepsilon\rightarrow 0}S_{\varepsilon},
\end{align*}
where
$$S_{\varepsilon}:=\int_{\varepsilon}^{1/\varepsilon}\left(\sum_{[Na,b,c]\in\mathcal{Q}_{N,D,\rho}}\pm\sum_{[Na,b,c]\in\mathcal{Q}_{N,D,\rho}'}\right)\frac{t^{n}}{(Na(it)^{2}+bit+c)^{k}}dt.$$
Since
\begin{align*}
    \sum_{[Na,b,c]\in\mathcal{Q}_{N,D,\rho}'}\int_{\varepsilon}^{1/\varepsilon}\frac{t^{n}}{(-Nat^{2}+bit+c)^{k}}dt&=\sum_{[Na,b,c]\in\mathcal{Q}_{N,D,\rho}}\int_{\varepsilon}^{1/\varepsilon}\frac{t^{n}}{(-Nat^{2}-bit+c)^{k}}dt\\
    &=(-1)^{n}\sum_{[Na,b,c]\in\mathcal{Q}_{N,D,\rho}}\int_{-1/\varepsilon}^{-\varepsilon}\frac{t^{n}}{(-Nat^{2}+bit+c)^{k}}dt,
\end{align*}
we have
\begin{align*}
    S_{\varepsilon}&=\sum_{[Na,b,c]\in\mathcal{Q}_{N,D,\rho}}\left(\int_{\varepsilon}^{1/\varepsilon}+\int_{-1/\varepsilon}^{-\varepsilon}\right)\frac{t^{n}}{(-Nat^{2}+bit+c)^{k}}dt\\
    &=\sum_{[Na,b,c]\in\mathcal{Q}_{N,D,\rho}}\left(\int_{-\infty}^{\infty}-\int_{-\infty}^{-1/\varepsilon}-\int_{-\varepsilon}^{\varepsilon}-\int_{1/\varepsilon}^{\infty}\right)\frac{t^{n}}{(-Nat^{2}+bit+c)^{k}}dt.
\end{align*}
Since $\mathcal{Q}_{N,D,\rho}\circ W_{N}=\mathcal{Q}_{N,D,-\rho}$,
\begin{align*}
    \sum_{[Na,b,c]\in\mathcal{Q}_{N,D,\rho}}&\int_{-\infty}^{-1/\varepsilon}\frac{t^{n}}{(-Nat^{2}+bit+c)^{k}}dt\\
    &=\sum_{[Na,b,c]\in\mathcal{Q}_{N,D,\rho}}\int_{0}^{-\varepsilon/N}\frac{N^{-n}s^{-n}}{(-aN^{-1}s^{-2}+biN^{-1}s^{-1}+c)^{k}}\cdot\left(-\frac{1}{Ns^{2}}\right)ds\\
    &=\sum_{[Na,b,c]\in\mathcal{Q}_{N,D,\rho}}\int_{-\varepsilon/N}^{0}\frac{(-1)^{k}N^{k-1-n}s^{2k-2-n}}{(-Ncs^{2}-bis+a)^{k}}ds\\
    &=\sum_{[Na,b,c]\in\mathcal{Q}_{N,D,-\rho}}\int_{-\varepsilon/N}^{0}\frac{(-1)^{k}N^{k-1-n}t^{2k-2-n}}{(-Nat^{2}+bit+c)^{k}}dt.
\end{align*}
Similarly,
$$\sum_{[Na,b,c]\in\mathcal{Q}_{N,D,\rho}}\int_{1/\varepsilon}^{\infty}\frac{t^{n}}{(-Nat^{2}+bit+c)^{k}}dt=\sum_{[Na,b,c]\in\mathcal{Q}_{N,D,-\rho}}\int_{0}^{\varepsilon/N}\frac{(-1)^{k}N^{k-1-n}t^{2k-2-n}}{(-Nat^{2}+bit+c)^{k}}dt.$$
Hence,
$$S_{\varepsilon}=S_{1}+S_{\varepsilon}'+S_{\varepsilon}'',$$
where
\begin{align*}
    S_{1}&=\sum_{[Na,b,c]\in\mathcal{Q}_{N,D,\rho}}\int_{-\infty}^{\infty}\frac{t^{n}}{(-Nat^{2}+bit+c)^{k}}dt,\\
    S_{\varepsilon}'&=-\sum_{[Na,b,c]\in\mathcal{Q}_{N,D,\rho}}\int_{-\varepsilon}^{\varepsilon}\frac{t^{n}}{(-Nat^{2}+bit+c)^{k}}dt,\\
    S_{\varepsilon}''&=-\sum_{[Na,b,c]\in\mathcal{Q}_{N,D,-\rho}}\int_{-\varepsilon/N}^{\varepsilon/N}\frac{(-1)^{k}N^{k-1-n}t^{2k-2-n}}{(-Nat^{2}+bit+c)^{k}}dt.
\end{align*}

We claim that
$$S_{1}=\sum\limits_{\substack{[Na,b,c]\in\mathcal{Q}_{N,D,\rho} \\ ac<0}}\int_{-\infty}^{\infty}\frac{t^{n}}{(-Nat^{2}+bit+c)^{k}}dt.$$
To prove this, let $Q=[Na,b,c]\in\mathcal{Q}_{N,D,\rho}$. Then $ac\neq 0$ because $D$ is not a perfect square. Assume now that $ac>0$, and set
$$P_{Q}(t):=-Nat^{2}+bit+c.$$
Then the zeros of $P_{Q}$ are given by
$$t_{\pm}=\frac{bi\pm\sqrt{-D}}{2Na}=\frac{i(b\pm\sqrt{D})}{2Na}.$$
In particular, $t_{\pm}$ are purely imaginary.
Moreover,
$$\mathrm{Im}(t_{+})\mathrm{Im}(t_{-})
=\frac{(b+\sqrt{D})(b-\sqrt{D})}{(2Na)^{2}}
=\frac{b^{2}-D}{(2Na)^{2}}
=\frac{4Nac}{(2Na)^{2}}.$$
Since $ac>0$, the poles $t_{\pm}$ of the function
$$\frac{t^{n}}{P_{Q}(t)^{k}}$$
lie on the same side of the real axis.
Let $C_{R}$ be the semicircular arc centered at $0$ of radius $R>0$ lying on the opposite side, and let $\Gamma_R=[-R,R]\cup C_{R}$.
By Cauchy's theorem,
$$\int_{\Gamma_{R}}\frac{t^{n}}{P_{Q}(t)^{k}}dt=\int_{-R}^{R}\frac{t^{n}}{P_{Q}(t)^{k}}dt+\int_{C_{R}}\frac{t^{n}}{P_{Q}(t)^{k}}dt=0.$$
Since $n\leq 2k-2$, and $\frac{t^{n}}{P_{Q}(t)^{k}}=O(|t|^{n-2k})$ as $|t|\rightarrow \infty$,
$$\lim_{R\rightarrow\infty}\int_{C_{R}}\frac{t^{n}}{P_{Q}(t)^{k}}dt=0,$$
which implies that
$$\int_{-\infty}^{\infty}\frac{t^{n}}{(-Nat^{2}+bit+c)^{k}}dt=0.$$
Consequently,
$$S_{1}=\sum_{\substack{[Na,b,c]\in\mathcal{Q}_{N,D,\rho}\\ ac<0}}
\int_{-\infty}^{\infty}\frac{t^{n}}{(-Nat^{2}+bit+c)^{k}}dt$$
as claimed.

Arguing as in the proof of \cite[Theorem 5]{KZ84}, we conclude that the integral in $S_{1}$ is equal
$$2i^{-n}c_{k,D}\binom{2k-2}{n}^{-1}\mathrm{sgn}(a)d_{k,n}(c,-b,Na).$$
Here, $d_{k,n}(\alpha,\beta,\gamma)$ denotes the coefficient of $X^{n}$ in $(\alpha X^{2}+\beta X+\gamma)^{k-1}$.
Thus, the contribution of $S_{1}$ to $r_{n,f_{k,N,D,\rho}^{\pm}}(I)$ is
$$(-1)^{[n/2]}\binom{2k-2}{n}^{-1}\left(\sum\limits_{\substack{[Na,b,c]\in\mathcal{Q}_{N,D,\rho} \\ a>0>c}}d_{k,n}(c,-b,Na)-\sum\limits_{\substack{[Na,b,c]\in\mathcal{Q}_{N,D,\rho} \\ a<0<c}}d_{k,n}(c,-b,Na)\right).$$
Note that for each fixed $[Na,b,c]\in\mathcal{Q}_{N,D,\rho}$ we have
$$(NaX^{2}-bX+c)^{k-1}=X^{2k-2}(cX^{-2}-bX^{-1}+Na)^{k-1}=\sum_{n=0}^{2k-2}d_{k,n}(c,-b,Na)X^{2k-2-n}.$$
Hence, summing over $Q\in\mathcal{Q}_{N,D,\rho}$ in the two ranges $a>0>c$ and $a<0<c$ and comparing the coefficients of $X^{2k-2-n}$, we see that the contribution of $S_{1}$ to
$r_{f_{k,N,D,\rho}^{+}}^{+}(I)+r_{f_{k,N,D,\rho}^{-}}^{-}(I)$ is
$$\sum\limits_{\substack{[Na,b,c]\in\mathcal{Q}_{N,D,\rho} \\ a>0>c}}(NaX^{2}-bX+c)^{k-1}-\sum\limits_{\substack{[Na,b,c]\in\mathcal{Q}_{N,D,\rho} \\ a<0<c}}(NaX^{2}-bX+c)^{k-1}.$$

Next, we compute $\lim_{\varepsilon\rightarrow 0}S_{\varepsilon}'$. Recall that
$$S_{\varepsilon}'=-\sum_{[Na,b,c]\in\mathcal{Q}_{N,D,\rho}}\int_{-\varepsilon}^{\varepsilon}\frac{t^{n}}{(-Nat^{2}+bit+c)^{k}}dt.$$
Since $D$ is not a perfect square, the set $\mathcal{Q}_{N,D,\rho}$ contains no forms $[Na,b,c]$ with $c=0$. 
Thus,
\begin{align*}
    S_{\varepsilon}'&=-\sum\limits_{\substack{[Na,b,c]\in\mathcal{Q}_{N,D,\rho} \\ c>0}}\int_{-\varepsilon}^{\varepsilon}\frac{t^{n}}{(-Nat^{2}+bit+c)^{k}}dt-\sum\limits_{\substack{[Na,b,c]\in\mathcal{Q}_{N,D,\rho} \\ c<0}}\int_{-\varepsilon}^{\varepsilon}\frac{t^{n}}{(-Nat^{2}+bit+c)^{k}}dt.
\end{align*}
If $Q=[Na,b,c]\in\mathcal{Q}_{N,D,\rho}$ and $c<0$, then $-Q=[N(-a),-b,-c]\in\mathcal{Q}_{N,D,-\rho}$ and $-c>0$. Moreover,
$$-N(-a)t^{2}+i(-b)t+(-c)=-(-Nat^{2}+bit+c),$$
so
$$\frac{t^{n}}{(-Nat^{2}+bit+c)^{k}}=(-1)^{k}\frac{t^{n}}{(-N(-a)t^{2}+i(-b)t+(-c))^{k}}.$$
Consequently,
\begin{equation*}
    \begin{split}
        S_{\varepsilon}'&=-\sum\limits_{\substack{[Na,b,c]\in\mathcal{Q}_{N,D,\rho} \\ c>0}}\int_{-\varepsilon}^{\varepsilon}\frac{t^{n}}{(-Nat^{2}+bit+c)^{k}}dt\\
        &\qquad\qquad\qquad\qquad\qquad-(-1)^{k}\sum\limits_{\substack{[Na,b,c]\in\mathcal{Q}_{N,D,-\rho} \\ c>0}}\int_{-\varepsilon}^{\varepsilon}\frac{t^{n}}{(-Nat^{2}+bit+c)^{k}}dt\\
        &=-\sum_{\mathcal{C}\in\mathcal{Q}_{N,D,\rho}/\Gamma_{0}(N)}\sum\limits_{\substack{[Na,b,c]\in\mathcal{C} \\ c>0}}\int_{-\varepsilon}^{\varepsilon}\frac{t^{n}}{(-Nat^{2}+bit+c)^{k}}dt\\
        &\qquad\qquad\qquad\qquad\qquad-(-1)^{k}\sum_{\mathcal{C}\in\mathcal{Q}_{N,D,-\rho}/\Gamma_{0}(N)}\sum\limits_{\substack{[Na,b,c]\in\mathcal{C} \\ c>0}}\int_{-\varepsilon}^{\varepsilon}\frac{t^{n}}{(-Nat^{2}+bit+c)^{k}}dt.
    \end{split}
\end{equation*}
Applying Lemma \ref{lem-limit} and letting $\varepsilon\rightarrow 0^{+}$, we obtain
\begin{equation*}
    \begin{split}
        \lim_{\varepsilon\rightarrow 0^{+}}S_{\varepsilon}'&=-\delta_{n,0}\frac{2\pi}{N(2k-1)}\cdot\frac{1}{\zeta(2k)\prod_{p\mid N}(1-p^{-2k})}\\
        &\qquad\qquad\qquad\times\left(\sum_{\mathcal{C}\in\mathcal{Q}_{N,D,\rho}/\Gamma_{0}(N)}\zeta_{\mathcal{C}}^{(N)}(k)+(-1)^{k}\sum_{\mathcal{C}\in\mathcal{Q}_{N,D,-\rho}/\Gamma_{0}(N)}\zeta_{\mathcal{C}}^{(N)}(k)\right).
    \end{split}
\end{equation*}

It remains to calculate $\lim_{\varepsilon\rightarrow 0^{+}}S_{\varepsilon}''$.
Putting $m=2k-2-n$ and $\eta=\varepsilon/N$, and arguing as in the computation of $\lim_{\varepsilon\rightarrow 0^{+}}S_{\varepsilon}'$, we have
\begin{align*}
    S_{\varepsilon}''&=-(-1)^{k}N^{k-1-n}\sum\limits_{\substack{[Na,b,c]\in\mathcal{Q}_{N,D,-\rho} \\ c>0}}\int_{-\eta}^{\eta}\frac{t^{m}}{(-Nat^{2}+bit+c)^{k}}dt\\
    &\qquad\qquad\qquad-(-1)^{k}N^{k-1-n}\sum\limits_{\substack{[Na,b,c]\in\mathcal{Q}_{N,D,-\rho} \\ c<0}}\int_{-\eta}^{\eta}\frac{t^{m}}{(-Nat^{2}+bit+c)^{k}}dt\\
    &=-(-1)^{k}N^{k-1-n}\sum\limits_{\substack{[Na,b,c]\in\mathcal{Q}_{N,D,-\rho} \\ c>0}}\int_{-\eta}^{\eta}\frac{t^{m}}{(-Nat^{2}+bit+c)^{k}}dt\\
    &\qquad\qquad\qquad-N^{k-1-n}\sum\limits_{\substack{[Na,b,c]\in\mathcal{Q}_{N,D,\rho} \\ c>0}}\int_{-\eta}^{\eta}\frac{t^{m}}{(-Nat^{2}+bit+c)^{k}}dt\\
    &=-N^{k-1-n}\sum_{\mathcal{C}\in\mathcal{Q}_{N,D,\rho}/\Gamma_{0}(N)}\sum\limits_{\substack{[Na,b,c]\in\mathcal{C} \\ c>0}}\int_{-\eta}^{\eta}\frac{t^{m}}{(-Nat^{2}+bit+c)^{k}}dt\\
    &\qquad\qquad\qquad-(-1)^{k}N^{k-1-n}\sum_{\mathcal{C}\in\mathcal{Q}_{N,D,-\rho}/\Gamma_{0}(N)}\sum\limits_{\substack{[Na,b,c]\in\mathcal{C} \\ c>0}}\int_{-\eta}^{\eta}\frac{t^{m}}{(-Nat^{2}+bit+c)^{k}}dt.
\end{align*}
Since $0\leq m\leq 2k-2$, one can apply Lemma \ref{lem-limit}. Using Lemma \ref{lem-limit} and $\delta_{m,0}=\delta_{n,2k-2}$, we get
\begin{align*}
    \lim_{\varepsilon\rightarrow 0^{+}}S_{\varepsilon}''&=-\delta_{m,0}\frac{2\pi N^{k-2-n}}{(2k-1)}\cdot\frac{1}{\zeta(2k)\prod_{p\mid N}(1-p^{-2k})}\\
    &\qquad\qquad\qquad\times\left(\sum_{\mathcal{C}\in\mathcal{Q}_{N,D,\rho}/\Gamma_{0}(N)}\zeta_{\mathcal{C}}^{(N)}(k)+(-1)^{k}\sum_{\mathcal{C}\in\mathcal{Q}_{N,D,-\rho}/\Gamma_{0}(N)}\zeta_{\mathcal{C}}^{(N)}(k)\right)\\
    &=-\delta_{n,2k-2}\frac{2\pi}{N^{k}(2k-1)}\cdot\frac{1}{\zeta(2k)\prod_{p\mid N}(1-p^{-2k})}\\
    &\qquad\qquad\qquad\times\left(\sum_{\mathcal{C}\in\mathcal{Q}_{N,D,\rho}/\Gamma_{0}(N)}\zeta_{\mathcal{C}}^{(N)}(k)+(-1)^{k}\sum_{\mathcal{C}\in\mathcal{Q}_{N,D,-\rho}/\Gamma_{0}(N)}\zeta_{\mathcal{C}}^{(N)}(k)\right).
\end{align*}
\end{proof}

\section{Proof of Theorem \ref{thm-zetaformula}}\label{sec-zetaformula}

In this section, we establish Theorem \ref{thm-zetaformula}.

\begin{lemma}\label{lem-frickeinv}
    Let $k,N,D$ and $\rho$ be as in Theorem \ref{thm-period}. Then
    $$f_{k,N,D,\rho}|_{2k}W_{N}=f_{k,N,D,-\rho}=f_{k,N,D,\rho}'.$$
\end{lemma}

\begin{proof}
    For any $Q(x,y)=[Na,b,c]\in\mathcal{Q}_{N,D,\rho}$,
    \begin{align*}
        Q\left(-\frac{1}{Nz},1\right)&=Na\cdot\left(-\frac{1}{Nz}\right)^{2}+b\cdot\left(-\frac{1}{Nz}\right)+c=(Nz)^{-2}\cdot(N^{2}cz^{2}-Nbz+Na)\\
        &=N^{-1}z^{-2}(Ncz^{2}-bz+a)=N^{-1}z^{-2}(Q\circ W_{N})(z,1).
    \end{align*}
    Thus,
    \begin{align*}
        (f_{k,N,D,\rho}|_{2k}W_{N})(z)&=(\sqrt{N}z)^{-2k}f_{k,N,D,\rho}\left(-\frac{1}{Nz}\right)=N^{-k}z^{-2k}f_{k,N,D,\rho}\left(-\frac{1}{Nz}\right)\\
        &=\frac{D^{k-1/2}}{\pi\binom{2k-2}{k-1}}\cdot\frac{1}{2}\sum_{Q\in\mathcal{Q}_{N,D,\rho}}N^{-k}z^{-2k}Q\left(-\frac{1}{Nz},1\right)^{-k}\\
        &=\frac{D^{k-1/2}}{\pi\binom{2k-2}{k-1}}\cdot\frac{1}{2}\sum_{Q\in\mathcal{Q}_{N,D,\rho}}N^{-k}z^{-2k}(N^{-1}z^{-2}(Q\circ W_{N})(z,1))^{-k}\\
        &=\frac{D^{k-1/2}}{\pi\binom{2k-2}{k-1}}\cdot\frac{1}{2}\sum_{Q\in\mathcal{Q}_{N,D,\rho}}(Q\circ W_{N})(z,1)^{-k}\\
        &=\frac{D^{k-1/2}}{\pi\binom{2k-2}{k-1}}\cdot\frac{1}{2}\sum_{Q\in\mathcal{Q}_{N,D,\rho}\circ W_{N}}Q(z,1)^{-k}.
    \end{align*}
    Since $\mathcal{Q}_{N,D,\rho}\circ W_{N}=\mathcal{Q}_{N,D,-\rho}$ by \eqref{eq-Fricke},
    $$(f_{k,N,D,\rho}|_{2k}W_{N})(z)=\frac{D^{k-1/2}}{\pi\binom{2k-2}{k-1}}\cdot\frac{1}{2}\sum_{Q\in\mathcal{Q}_{N,D,-\rho}}Q(z,1)^{-k}=f_{k,N,D,-\rho}(z),$$
    which proves the first equality.

    For the second equality, observe that
    \begin{align*}
        \mathcal{Q}_{N,D,\rho}^{'}&=\{[Na,-b,c]:[Na,b,c]\in\mathcal{Q}_{N,D,\rho}\}\\
        &=\{[Na,-b,c]:a,b,c\in\mathbb{Z},~b^{2}-4Nac=D,~b\equiv\rho\pmod{2N}\}\\
        &=\{[Na,b,c]:a,b,c\in\mathbb{Z},~b^{2}-4Nac=D,~b\equiv-\rho\pmod{2N}\}\\
        &=\mathcal{Q}_{N,D,-\rho}.
    \end{align*}
    This shows that
    \begin{align*}
        f_{k,N,D,\rho}'(z)&=\frac{D^{k-1/2}}{\pi\binom{2k-2}{k-1}}\cdot\frac{1}{2}\sum_{Q\in\mathcal{Q}_{N,D,\rho}'}Q(z,1)^{-k}=\frac{D^{k-1/2}}{\pi\binom{2k-2}{k-1}}\cdot\frac{1}{2}\sum_{Q\in\mathcal{Q}_{N,D,-\rho}}Q(z,1)^{-k}\\
        &=f_{k,N,D,-\rho}(z).
    \end{align*}
\end{proof}

Let $k$ and $N$ be positive integers. For $f=\sum_{n=1}^{\infty}a_{f}(n)q^{n}\in S_{2k}(N)$, the \textit{$L$-function of} $f$ is defined by the Dirichlet series
$$L(s,f):=\sum_{n=1}^{\infty}\frac{a_{f}(n)}{n^{s}}.$$
The series on the right-hand side converges absolutely for all complex numbers $s$ with $\mathrm{Re}(s)>k+1$, and defines a holomorphic function on the right half-plane $\mathrm{Re}(s)>k+1$. It is known that $L(s,f)$ has an analytic continuation to the whole complex plane. If $f|_{2k}W_{N}=f$, then the \textit{completed $L$-function}
$$\Lambda(s,f):=N^{s/2}(2\pi)^{-s}\Gamma(s)L(s,f)$$
satisfies the functional equation
\begin{equation}\label{eq-functionaleq}
    \Lambda(s,f)=(-1)^{k}\Lambda(2k-s,f).
\end{equation}
For details, see \cite[pp. 270--271]{K92}.

Now we define the \textit{scalar-valued period polynomial $p_{f}(X)$ of $f$} by
\begin{equation}\label{eq-svpp}
    p_{f}(X):=\int_{0}^{i\infty}f(z)(X-z)^{2k-2}dz.
\end{equation}
Write
$$p_{f}(X)=\sum_{n=0}^{2k-2}p_{n,f}X^{2k-2-n}.$$
Comparing \eqref{eq-svpp} with \eqref{eq-vvpp}, we have
$$p_{f}(X)=r_{f}(I)(X),$$
and hence
$$p_{n,f}=i^{-n+1}\binom{2k-2}{n}r_{n,f}(I)$$
for $n=0,1,\ldots,2k-2$. It is a known fact that
\begin{equation}\label{eq-periodlval}
    r_{n,f}(I)=\int_{0}^{\infty}f(it)t^{n}dt=N^{-(n+1)/2}\Lambda(n+1,f)
\end{equation}
for every $n\in\{0,1,\ldots,2k-2\}$; see \cite[p. 268]{K92} for instance. The following lemma provides us with an explicit relation between the coefficients $p_{n,f}$ and $p_{2k-2-n,f}$ ($n=0,1,\ldots,k-1$) for a cusp form $f\in S_{2k}(N)$ which is invariant under the action of Fricke involution $W_{N}$.

\begin{lemma}\label{lem-coeffrel}
    Let $k$ and $N$ be positive integers, and let $f\in S_{2k}(N)$. Suppose that $f|_{2k}W_{N}=f$. For $n=0,1,\ldots,k-1$,
    $$p_{n,f}=(-1)^{n+1}N^{k-1-n}p_{2k-2-n,f}.$$
\end{lemma}

\begin{proof}
    Fix $n\in\{0,1,\ldots,k-1\}$. By \eqref{eq-periodlval}, we have
    \begin{equation*}
        p_{n,f}=i^{-n+1}\binom{2k-2}{n}r_{n,f}(I)=i^{-n+1}\binom{2k-2}{n}N^{-(n+1)/2}\Lambda(n+1,f).
    \end{equation*}

    By \eqref{eq-periodlval} again,
    \begin{align*}
        p_{2k-2-n,f}&=i^{-(2k-2-n)+1}\binom{2k-2}{2k-2-n}N^{-(2k-2-n+1)/2}\Lambda(2k-2-n+1,f)\\
        &=(-1)^{k}i^{n+3}\binom{2k-2}{n}N^{-(2k-1-n)/2}\Lambda(2k-n-1,f).
    \end{align*}
    Since $f|_{2k}W_{N}=f$, one can apply the functional equation \eqref{eq-functionaleq} to obtain
    $$\Lambda(2k-n-1,f)=(-1)^{k}\Lambda(n+1,f).$$
    Therefore,
    \begin{align*}
        (-1)^{n+1}N^{k-1-n}p_{2k-2-n,f}&=i^{-2n-2}N^{k-1-n}p_{2k-2-n,f}\\
        &=i^{-2n-2}N^{k-1-n}\cdot i^{n+3}\binom{2k-2}{n}N^{-(2k-1-n)/2}\Lambda(n+1,f)\\
        &=i^{-n+1}\binom{2k-2}{n}N^{-(n+1)/2}\Lambda(n+1,f)\\
        &=p_{n,f}
    \end{align*}
    as desired.
\end{proof}

We are now ready to prove Theorem \ref{thm-zetaformula}.

\begin{proof}[Proof of Theorem \ref{thm-zetaformula}]
    For simplicity, we put
    $$r(X):=r_{f_{k,N,D,\rho}^{+}}^{+}(I)(X)+r_{f_{k,N,D,\rho}^{-}}^{-}(I)(X).$$
    Since $r_{f_{k,N,D,\rho}^{+}}^{+}(I)(-X)=r_{f_{k,N,D,\rho}^{+}}^{+}(I)(X)$ and $r_{f_{k,N,D,\rho}^{-}}^{-}(I)(-X)=-r_{f_{k,N,D,\rho}^{-}}^{-}(I)(X)$, the even-degree (resp. odd-degree) terms of $r(X)$ coincide with those of $r_{f_{k,N,D,\rho}^{+}}^{+}(I)(X)$ (resp. $r_{f_{k,N,D,\rho}^{-}}^{-}(I)(X)$). On the other hand, for any polynomial $P(X)$, the even part $P^{+}(X)=\frac{1}{2}(P(X)+P(-X))$ has the same even-degree coefficients as $P(X)$. Applying this to $P(X)=r_{f_{k,N,D,\rho}^{+}}^{+}(I)(X)$, we see that the even-degree terms of
    $r_{f_{k,N,D,\rho}^{+}}^{+}(I)(X)$ coincide with those of $r_{f_{k,N,D,\rho}^{+}}(I)(X)$.
    Since $p_{f}(X)=r_{f}(I)(X)$ for $f\in S_{2k}(N)$, it follows that the even-degree terms of
    $r_{f_{k,N,D,\rho}^{+}}^{+}(I)(X)$ coincide with those of $p_{f_{k,N,D,\rho}^{+}}(X)$. In particular,
    \begin{align*}
        p_{0,f_{k,N,D,\rho}^{+}}&=\text{the leading coefficient of}~r(X),\\
        p_{2k-2,f_{k,N,D,\rho}^{+}}&=\text{the constant term of}~r(X).
    \end{align*}

    Lemma \ref{lem-frickeinv} tells us that $f_{k,N,D,\rho}^{+}|_{2k}W_{N}=f_{k,N,D,\rho}^{+}$. It follows from Lemma \ref{lem-coeffrel} that
    $$p_{0,f_{k,N,D,\rho}^{+}}=-N^{k-1}p_{2k-2,f_{k,N,D,\rho}^{+}}.$$
    By Theorem \ref{thm-period},
    \begin{align*}
        p_{2k-2,f_{k,N,D,\rho}^{+}}&=\sum\limits_{\substack{[Na,b,c]\in\mathcal{Q}_{N,D,\rho} \\ a>0>c}}c^{k-1}-\sum\limits_{\substack{[Na,b,c]\in\mathcal{Q}_{N,D,\rho} \\ a<0<c}}c^{k-1}\\
        &\quad+\frac{\pi}{N^{k}(2k-1)c_{k,D}}\cdot\frac{\zeta_{N,D,-\rho}(k)+(-1)^{k}\zeta_{N,D,\rho}(k)}{\zeta(2k)\prod_{p\mid N}(1-p^{-2k})},
    \end{align*}
    and hence
    \begin{align*}
        p_{0,f_{k,N,D,\rho}^{+}}&=-N^{k-1}\left(\sum\limits_{\substack{[Na,b,c]\in\mathcal{Q}_{N,D,\rho} \\ a>0>c}}c^{k-1}-\sum\limits_{\substack{[Na,b,c]\in\mathcal{Q}_{N,D,\rho} \\ a<0<c}}c^{k-1}\right)\\
            &\quad-\frac{\pi}{N(2k-1)c_{k,D}}\cdot\frac{\zeta_{N,D,-\rho}(k)+(-1)^{k}\zeta_{N,D,\rho}(k)}{\zeta(2k)\prod_{p\mid N}(1-p^{-2k})}.
    \end{align*}
    Since $k$ is odd, we have
    \begin{equation}\label{eq-firstleading}
        \begin{split}
            p_{0,f_{k,N,D,\rho}^{+}}&=\frac{\pi}{N(2k-1)c_{k,D}}\cdot\frac{\zeta_{N,D,\rho}(k)-\zeta_{N,D,-\rho}(k)}{\zeta(2k)\prod_{p\mid N}(1-p^{-2k})}\\
            &\qquad\qquad\qquad\qquad-N^{k-1}\left(\sum\limits_{\substack{[Na,b,c]\in\mathcal{Q}_{N,D,\rho} \\ a>0>c}}c^{k-1}-\sum\limits_{\substack{[Na,b,c]\in\mathcal{Q}_{N,D,\rho} \\ a<0<c}}c^{k-1}\right).
        \end{split}
    \end{equation}
    
    Using Theorem \ref{thm-period} again, we get
    \begin{align*}
        p_{0,f_{k,N,D,\rho}^{+}}&=\sum\limits_{\substack{[Na,b,c]\in\mathcal{Q}_{N,D,\rho} \\ a>0>c}}(Na)^{k-1}-\sum\limits_{\substack{[Na,b,c]\in\mathcal{Q}_{N,D,\rho} \\ a<0<c}}(Na)^{k-1}\\
        &\quad -\frac{\pi}{N(2k-1)c_{k,D}}\cdot\frac{\zeta_{N,D,\rho}(k)+(-1)^{k}\zeta_{N,D,-\rho}(k)}{\zeta(2k)\prod_{p\mid N}(1-p^{-2k})}.
    \end{align*}
    Consider the map $\iota:\mathcal{Q}_{N,D,\rho}\rightarrow\mathcal{Q}_{N,D,\rho}$ defined by $\iota([Na,b,c])=[-Nc,b,-a]$. From the definition, one can check that $\iota$ is a well-defined map satisfying $\iota^{2}=\mathrm{id}$. Additionally, $\iota$ restricts to a bijection of the subset $\{[Na,b,c]\in\mathcal{Q}_{N,D,\rho}:a>0>c\}$, and $\iota$ also restricts to a bijection of the subset $\{[Na,b,c]\in\mathcal{Q}_{N,D,\rho}:a<0<c\}$. Hence,
    \begin{align*}
        p_{0,f_{k,N,D,\rho}^{+}}&=\sum\limits_{\substack{[Na,b,c]\in\mathcal{Q}_{N,D,\rho} \\ a>0>c}}(-Nc)^{k-1}-\sum\limits_{\substack{[Na,b,c]\in\mathcal{Q}_{N,D,\rho} \\ a<0<c}}(-Nc)^{k-1}\\
        &\quad -\frac{\pi}{N(2k-1)c_{k,D}}\cdot\frac{\zeta_{N,D,\rho}(k)+(-1)^{k}\zeta_{N,D,-\rho}(k)}{\zeta(2k)\prod_{p\mid N}(1-p^{-2k})}\\
        &=(-1)^{k-1}N^{k-1}\left(\sum\limits_{\substack{[Na,b,c]\in\mathcal{Q}_{N,D,\rho} \\ a>0>c}}c^{k-1}-\sum\limits_{\substack{[Na,b,c]\in\mathcal{Q}_{N,D,\rho} \\ a<0<c}}c^{k-1}\right)\\
        &\quad -\frac{\pi}{N(2k-1)c_{k,D}}\cdot\frac{\zeta_{N,D,\rho}(k)+(-1)^{k}\zeta_{N,D,-\rho}(k)}{\zeta(2k)\prod_{p\mid N}(1-p^{-2k})}.
    \end{align*}
    Since $k$ is odd, we obtain
    \begin{equation}\label{eq-secondleading}
        \begin{split}
            p_{0,f_{k,N,D,\rho}^{+}}&=\frac{\pi}{N(2k-1)c_{k,D}}\cdot\frac{\zeta_{N,D,-\rho}(k)-\zeta_{N,D,\rho}(k)}{\zeta(2k)\prod_{p\mid N}(1-p^{-2k})}\\
            &\qquad\qquad\qquad\qquad+N^{k-1}\left(\sum\limits_{\substack{[Na,b,c]\in\mathcal{Q}_{N,D,\rho} \\ a>0>c}}c^{k-1}-\sum\limits_{\substack{[Na,b,c]\in\mathcal{Q}_{N,D,\rho} \\ a<0<c}}c^{k-1}\right).
        \end{split}
    \end{equation}
    From \eqref{eq-secondleading}, \eqref{eq-firstleading}, $c_{k,D}=\binom{2k-2}{k-1}D^{1/2-k}\pi$ and the fact that $\zeta(2n)=\frac{2^{2n-1}B_{2n}\pi^{2n}}{(2n)!}$ for any positive odd integer $n$, we see that
    \begin{align*}
        \zeta_{N,D,\rho}(k)-\zeta_{N,D,-\rho}(k)&=\frac{2^{2k-1}N^{k}(2k-1)D^{1/2-k}B_{2k}\pi^{2k}}{(2k)!}\binom{2k-2}{k-1}\\
        &\qquad\qquad\times\prod_{p\mid N}(1-p^{-2k})\left(\sum\limits_{\substack{[Na,b,c]\in\mathcal{Q}_{N,D,\rho} \\ a>0>c}}c^{k-1}-\sum\limits_{\substack{[Na,b,c]\in\mathcal{Q}_{N,D,\rho} \\ a<0<c}}c^{k-1}\right).
    \end{align*}
\end{proof}

\section{Proof of Theorem \ref{thm-dedekindzeta}}\label{sec-dedekindzeta}

In this section, we prove Theorem \ref{thm-dedekindzeta}.
We first fix the notation.
Throughout this section, let $N$ and $D$ be as in the statement of Theorem \ref{thm-dedekindzeta}, and let $K=\mathbb{Q}(\sqrt{D})$. Since $D\equiv 1\pmod{4N}$, we have $D\equiv 1\pmod{4}$ and hence
$$\mathcal{O}_{K}=\mathbb{Z}\left[\omega_{0}\right],\quad\omega_{0}:=\frac{1+\sqrt{D}}{2}.$$
For an integral ideal $\mathfrak{a}\subset\mathcal{O}_{K}$, we denote by $\mathrm{Nm}(\mathfrak{a})$ the norm of $\mathfrak{a}$.
We write $\mathrm{Cl}^{+}(\mathcal{O}_{K})$ for the narrow ideal class group of $K$.

For a primitive form $Q=[A,B,C]\in\mathcal{Q}_D^{0}$, set
$$\omega_{B}:=\frac{-B+\sqrt{D}}{2}\in\mathcal{O}_{K},\quad\alpha_{Q}:=\frac{-B+\sqrt{D}}{2A},\quad\alpha_{Q}':=\frac{-B-\sqrt{D}}{2A},$$
and define the associated ideal
$$\mathfrak{a}_{Q}:=(A,\omega_{B})=\left(A,\frac{-B+\sqrt{D}}{2}\right)\subset\mathcal{O}_{K}.$$
Note that $\omega_{B}=A\alpha_{Q}$ and $\alpha_{Q}'\omega_{B}=C$.

Next, we collect several auxiliary results needed in our proof.

\begin{lemma}\label{lem-pairtoideal}
Let $Q=[A,B,C]\in\mathcal{Q}_{D}^{0}$.
For $(u,v)\in\mathbb{Z}^{2}$, define
$$\mathfrak{b}_{u,v}(Q):=(u-v\alpha_{Q}')\mathfrak{a}_{Q}.$$
Then
\begin{enumerate}
    \item[\textnormal{(i)}] If $Q(u,v)>0$, then $\mathfrak{b}_{u,v}(Q)$ is an integral ideal of $\mathcal{O}_{K}$.
    \item[\textnormal{(ii)}] One has $\mathrm{Nm}(\mathfrak{b}_{u,v}(Q))=Q(u,v)$.
    \item[\textnormal{(iii)}] The map $(u,v)\mapsto\mathfrak{b}_{u,v}(Q)$ induces a well-defined bijection between
    \begin{equation}\label{eq-pairideal}
        \{(u,v)\in\mathbb{Z}^2/\mathrm{SL}_{2}(\mathbb{Z})_{Q}:\ Q(u,v)>0\}\quad\text{and}\quad
        \{\mathfrak{b}\subset\mathcal{O}_{K}:[\mathfrak{b}]=[\mathfrak{a}_{Q}]~\text{in}~\mathrm{Cl}^{+}(\mathcal{O}_{K})\}.
    \end{equation}
\end{enumerate}
\end{lemma}

\begin{proof}
Note that
$$A\alpha_{Q}'=\omega_{B}'=\frac{-B-\sqrt{D}}{2}\quad\text{and}\quad\alpha_{Q}'\omega_{B}=\frac{(-B-\sqrt{D})(-B+\sqrt{D})}{4A}=\frac{B^{2}-D}{4A}=C\in\mathbb{Z}.$$

(i) We compute
$$\mathfrak{b}_{u,v}(Q)=(u-v\alpha_{Q}')(A,\omega_{B})=(uA-vA\alpha_{Q}',u\omega_{B}-v\alpha_{Q}'\omega_{B})=(uA-v\omega_{B}',u\omega_{B}-vC).$$
Since $\omega_{B},\omega_{B}'\in\mathcal{O}_{K}$, we conclude that $\mathfrak{b}_{u,v}(Q)$ is an integral ideal.

(ii) Since $\alpha_Q$ and $\alpha_Q'$ are conjugates, we have
$$\mathrm{Nm}(u-v\alpha_{Q}')=(u-v\alpha_{Q}')(u-v\alpha_{Q}).$$
Using $\alpha_{Q}+\alpha_{Q}'=-B/A$ and $\alpha_{Q}\alpha_{Q}'=C/A$, we obtain
$$\mathrm{Nm}(u-v\alpha_{Q}')=u^{2}-uv(\alpha_{Q}+\alpha_{Q}')+v^{2}\alpha_{Q}\alpha_{Q}'=u^{2}+\frac{B}{A}uv+\frac{C}{A}v^{2}=\frac{Q(u,v)}{A}.$$
Moreover, $\mathrm{Nm}(\mathfrak{a}_{Q})=A$. Therefore,
$$\mathrm{Nm}(\mathfrak{b}_{u,v}(Q))=\mathrm{Nm}(u-v\alpha_{Q}')\cdot\mathrm{Nm}(\mathfrak{a}_{Q})=Q(u,v).$$

(iii) If $Q(u,v)>0$, then $\mathrm{Nm}(u-v\alpha_Q')=Q(u,v)/A>0$.
In a real quadratic field, a principal ideal generated by an element of positive norm is narrow principal (replace the generator by its negative if necessary).
Hence, $[(u-v\alpha_{Q}')\mathfrak{a}_{Q}]=[\mathfrak{a}_{Q}]$ in $\mathrm{Cl}^{+}(\mathcal{O}_{K})$.

We next check that the map is well-defined on $\mathbb{Z}^{2}/\mathrm{SL}_{2}(\mathbb{Z})_{Q}$.
Let $M=\left(\begin{smallmatrix} \alpha & \beta \\ \gamma & \delta\end{smallmatrix}\right)\in\mathrm{SL}_{2}(\mathbb{Z})_{Q}$ and set $(u',v'):=(u,v)M$.
Then $M$ fixes both $\alpha_{Q}$ and $\alpha_{Q}'$ under the fractional linear transformation, and $\mathrm{Nm}(\epsilon_{M})=1$, where $\epsilon_{M}=\gamma\alpha_{Q}'+\delta$.
Moreover, one has
$$u'-v'\alpha_{Q}'=\frac{u-v\alpha_{Q}'}{\gamma\alpha_{Q}'+\delta}=\frac{u-v\alpha_{Q}'}{\epsilon_{M}}.$$
Since $(\epsilon_{M})=\mathcal{O}_{K}$,
$$\mathfrak{b}_{u',v'}(Q)=(u'-v'\alpha_{Q}')\mathfrak{a}_{Q}=(u-v\alpha_{Q}')\mathfrak{a}_{Q}=\mathfrak{b}_{u,v}(Q).$$
Thus, the map is well-defined on the quotient.

For surjectivity, let $\mathfrak{b}\subset\mathcal{O}_{K}$ satisfy $[\mathfrak{b}]=[\mathfrak{a}_{Q}]$ in $\mathrm{Cl}^{+}(\mathcal{O}_{K})$.
Then there exists totally positive $\beta\in K^{\times}$ with $\mathfrak{b}=(\beta)\mathfrak{a}_{Q}$.
Since $\mathfrak{b}\subset\mathcal{O}_{K}$, we have $\beta\in\mathfrak{a}_{Q}^{-1}$.
One checks that
$$\mathfrak{a}_{Q}^{-}=\mathbb{Z}+\mathbb{Z}\alpha_{Q}'.$$
Indeed, $1\in\mathfrak{a}_{Q}^{-1}$ and $\alpha_{Q}'\in\mathfrak{a}_{Q}^{-1}$ since $\alpha_Q'A=\omega_{B}'\in\mathcal{O}_{K}$ and $\alpha_Q'\omega_B=C\in\mathbb{Z}$.
$$(A,\omega_{B})(\mathbb{Z}+\mathbb{Z}\alpha_{Q}')=(A,\omega_{B},A\alpha_{Q}',\omega_{B}\alpha_{Q}')=(A,\omega_{B},\omega_{B}',C)$$
contains $B=-(\omega_{B}+\omega_{B}')\in\mathbb{Z}$, hence contains $(A,B,C)=\mathcal{O}_{K}$ because $\gcd(A,B,C)=1$.
Therefore, $\mathfrak{a}_{Q}^{-1}=\mathbb{Z}+\mathbb{Z}\alpha_{Q}'$. This shows that $\beta=u-v\alpha_{Q}'$ for some $u,v\in\mathbb{Z}$, and $\mathfrak{b}=\mathfrak{b}_{u,v}(Q)$.
Finally, by (ii), 
$$0<\mathrm{Nm}(\mathfrak b)=\mathrm{Nm}(\mathfrak{b}_{u,v}(Q))=Q(u,v).$$

Injectivity follows similarly: if $\mathfrak b_{u,v}(Q)=\mathfrak b_{u',v'}(Q)$, then
$$\frac{u-v\alpha_{Q}'}{u'-v'\alpha_{Q}'}\in\mathcal{O}_{K}'\quad\text{and}\quad\mathrm{Nm}\left(\frac{u-v\alpha_{Q}'}{u'-v'\alpha_{Q}'}\right)=.$$
Such a unit corresponds to an element of $\mathrm{PSL}_{2}(\mathbb{Z})_{Q}$; if necessary, composing with $-I\in\mathrm{SL}_{2}(\mathbb{Z})_{Q}$ gives an element of $\mathrm{SL}_2(\mathbb{Z})_{Q}$.
Hence, $(u,v)$ and $(u',v')$ represent the same class in $\mathbb{Z}^{2}/\mathrm{SL}_{2}(\mathbb{Z})_{Q}$.
\end{proof}

We next compare the orbits $\mathbb{Z}^{2}/\Gamma_{0}(N)_{Q}$ and $\mathbb{Z}^{2}/\mathrm{SL}_2(\mathbb{Z})_{Q}$.

\begin{lemma}\label{lem-stabilizer}
    Let $Q\in\mathcal{Q}_{N,D,1}$. Then the map
    \begin{align*}
        \mathbb{Z}^{2}/\Gamma_{0}(N)_{Q}=\mathbb{Z}^{2}/(\mathrm{SL}_{2}(\mathbb{Z})_{Q}\cap\Gamma_{0}(N))&\rightarrow\mathbb{Z}^{2}/\mathrm{SL}_{2}(\mathbb{Z})_{Q}\\
        [(u,v)]&\mapsto[(u,v)]
    \end{align*}
    is surjective, and the inverse image of each element of $\mathbb{Z}^{2}/\mathrm{SL}_{2}(\mathbb{Z})_{Q}$ has cardinality
    $$e_{Q}:=[\mathrm{SL}_{2}(\mathbb{Z})_{Q}:\mathrm{SL}_{2}(\mathbb{Z})_{Q}\cap\Gamma_{0}(N)].$$
\end{lemma}

To decompose the sum over $(u,v)$ according to $(v,N)$ later, we need the following lemma.

\begin{lemma}\label{lem}
    Let $Q=[Na,b,c]\in\mathcal{Q}_{N,D,1}$ and let $d$ be a positive divisor of $N$.
    Put
    \[
    Q^{(d)}:=\left[\frac{N}{d}a,b,cd\right]\in\mathcal{Q}_{N/d,D,1}
    \]
    and set
    \[
    \Gamma_{0}^{0}(N/d,d):=\left\{\left(\begin{smallmatrix}
        \alpha & \beta \\ \gamma & \delta
    \end{smallmatrix}\right)\in\Gamma_{0}(N/d):\beta\equiv 0\pmod{d}\right\}.
    \]
    Then the map
    $$\mathbb{Z}^{2}\rightarrow\mathbb{Z}^{2},\quad(u,v)\mapsto(u,dv)$$
    induces a one-to-one correspondence between
    $$\{(u,v)\in\mathbb{Z}^{2}/\Gamma_{0}^{0}(N/d,d)_{Q^{(d)}}:Q^{(d)}(u,v)>0,~(v,N/d)=1\}$$
    and
    $$\{(u,v)\in\mathbb{Z}^{2}/\Gamma_{0}(N)_{Q}:Q(u,v)>0,~(v,N)=d\}.$$
\end{lemma}

The next lemma is a congruence property for stabilizers of primitive forms.

\begin{lemma}\label{lem-stabform}
    Let $Q=[A,B,C]\in\mathcal{Q}_{D}^{0}$. Then every $M=\left(\begin{smallmatrix}
        \alpha & \beta \\ \gamma & \delta
    \end{smallmatrix}\right)\in\mathrm{SL}_{2}(\mathbb{Z})_{Q}$ satisfies $\beta\equiv 0\pmod{C}$ and $\gamma\equiv 0\pmod{A}$.
\end{lemma}

\begin{proof}
    Let $M=\left(\begin{smallmatrix}
        \alpha & \beta \\ \gamma & \delta
    \end{smallmatrix}\right)\in\mathrm{SL}_{2}(\mathbb{Z})_{Q}$. It is known that
    $$\alpha=\frac{t-Bu}{2},\quad\beta=-Cu,\quad\gamma=Au\quad\text{and}\quad\delta=\frac{t+Bu}{2}$$
    for some integers $t,u$ satisfying $t^{2}-Du^{2}=4$ (see \cite[pp. 31--32]{B89}). In particular,
    $$\beta=-Cu\equiv 0\pmod{C}\quad\text{and}\quad\gamma=Au\equiv 0\pmod{A}.$$
\end{proof}

\begin{corollary}\label{cor-indexone}
    Let $d$ be any positive divisor of $N$ and let $Q=[Na,b,c]\in\mathcal{Q}_{N,D,1}$. If we put
    $$Q^{(d)}=\left[\frac{N}{d}a,b,cd\right],$$
    then one has
    $$\Gamma_{0}(N/d)_{Q^{(d)}}=\Gamma_{0}^{0}(N/d,d)_{Q^{(d)}},$$
    and hence
    \begin{equation}\label{eq-indexone}
        [\Gamma_{0}(N/d)_{Q^{(d)}}:\Gamma_{0}^{0}(N/d,d)_{Q^{(d)}}]=1.
    \end{equation}
\end{corollary}

\begin{proof}
    It suffices to show that
    $$\Gamma_{0}(N/d)_{Q^{(d)}}\subset\Gamma_{0}^{0}(N/d,d).$$
    Let $M=\left(\begin{smallmatrix}
        \alpha & \beta \\ \gamma & \delta
    \end{smallmatrix}\right)\in\Gamma_{0}(N/d)_{Q^{(d)}}$. Then $M\in\mathrm{SL}_{2}(\mathbb{Z})_{Q^{(d)}}$. By Lemma \ref{lem-stabform},
    $$\beta\equiv 0\pmod {cd}\quad\text{and}\quad\gamma\equiv 0\pmod{Na/d}.$$
    This implies that
    $$\beta\equiv 0\pmod{d}\quad\text{and}\quad\gamma\equiv 0\pmod{N/d}.$$
\end{proof}

In the following lemma, we give a factorization of $\mathcal{a}_{Q}$.

\begin{lemma}\label{lem-idealtranslation}
Let $Q=[A,B,C]\in\mathcal{Q}_{D}^{0}$ and set
\[
\mathfrak{a}_{Q}:=\left(A,\frac{-B+\sqrt{D}}{2}\right)\subset\mathcal{O}_{K}.
\]
Assume that $d\mid A$ and $\gcd(B,d)=1$. Put
\[
\omega_{B}:=\frac{-B+\sqrt{D}}{2}.
\]
Then
\begin{equation}\label{eq-idealfactor}
\mathfrak{a}_{Q}=(A,\omega_{B})=\left(\frac{A}{d},\omega_{B}\right)\cdot (d,\omega_{B}).
\end{equation}
In particular,
\begin{equation}\label{eq-idealclasstranslation}
\left[\left(\frac{A}{d},\omega_{B}\right)\right]=[\mathfrak{a}_{Q}]\cdot[(d,\omega_{B})]^{-1}
\quad\text{in }\mathrm{Cl}^{+}(\mathcal{O}_{K}).
\end{equation}
\end{lemma}

\begin{proof}
Set
\[
I:=\left(\frac{A}{d},\omega_{B}\right),\qquad J:=(d,\omega_{B}).
\]
First, we show $IJ\subset (A,\omega_{B})$.
Indeed, $IJ$ is generated by
\[
\frac{A}{d}\cdot d=A,\qquad \frac{A}{d}\omega_{B},\qquad d\omega_{B},\qquad \omega_{B}^{2}.
\]
Clearly $A\in (A,\omega_{B})$, $\frac{A}{d}\omega_{B}\in (A,\omega_{B})$, and $d\omega_{B}\in (A,\omega_{B})$.
Moreover, since $D=B^{2}-4AC$, one has
\[
\omega_{B}^{2}+B\omega_{B}+AC=0,
\]
hence $\omega_{B}^{2}\in (A,\omega_{B})$.
Therefore $IJ\subset (A,\omega_{B})$.

Next, we show $(A,\omega_{B})\subset IJ$.
Since $A=\frac{A}{d}\cdot d\in IJ$, it remains to show $\omega_{B}\in IJ$.
Choose integers $x,y$ such that $xB+yd=1$.
Then
\begin{equation}\label{eq-omegalinear}
\omega_{B}=xB\omega_{B}+yd\,\omega_{B}.
\end{equation}
Here $d\omega_{B}\in IJ$ by definition.
Also, using $\omega_{B}^{2}+B\omega_{B}+AC=0$, we obtain
\[
B\omega_{B}=-\omega_{B}^{2}-AC.
\]
Since $\omega_{B}^{2}\in IJ$ (because $\omega_{B}\in I$ and $\omega_{B}\in J$) and $AC\in IJ$ (because $A\in IJ$),
we have $B\omega_{B}\in IJ$.
Thus, the right-hand side of \eqref{eq-omegalinear} lies in $IJ$, and hence $\omega_{B}\in IJ$.
Therefore $(A,\omega_{B})\subset IJ$.

Combining the two inclusions gives \eqref{eq-idealfactor}, and \eqref{eq-idealclasstranslation} follows immediately.
\end{proof}

With these auxiliary results in place, we now proceed to the proof of Theorem \ref{thm-dedekindzeta}.

\begin{proof}[Proof of Theorem \ref{thm-dedekindzeta}]
From the proposition given in \cite[p. 505]{GKZ87}, we see that the inclusion map $\mathcal{Q}_{N,D,1}\hookrightarrow\mathcal{Q}_{D}^{0}$ induces a bijection
\begin{equation}\label{eq-gkzbij}
    \mathcal{Q}_{N,D,1}/\Gamma_{0}(N)\xrightarrow{\sim}\mathcal{Q}_{D}^{0}/\mathrm{SL}_{2}(\mathbb{Z}).
\end{equation}

Since $D\equiv 1\pmod{4N}$, we have $D\equiv 1\pmod{4}$, and hence
$$\mathcal{O}_{K}=\mathbb{Z}\left[\frac{1+\sqrt{D}}{2}\right].$$
We denote by $\mathrm{Cl}^{+}(\mathcal{O}_{K})$ the narrow ideal class group of $\mathcal{O}_{K}$.
It is known that the map
$$Q=[A,B,C]\mapsto\mathfrak{a}_{Q}:=\left(A,\frac{-B+\sqrt{D}}{2}\right)\subset\mathcal{O}_{K}$$
induces a one-to-one correspondence
$$\mathcal{Q}_{D}^{0}/\mathrm{SL}_{2}(\mathbb{Z})\xrightarrow{\sim} \mathrm{Cl}^{+}(\mathcal{O}_{K}).$$
(see, for example, \cite{FT93}). Composing this with \eqref{eq-gkzbij} gives a map
$$\Theta:\ \mathcal{Q}_{N,D,1}/\Gamma_{0}(N)\xrightarrow{\sim} \mathrm{Cl}^{+}(\mathcal{O}_{K}).$$

Now let $\mathcal{C}\in\mathcal{Q}_{N,D,1}/\Gamma_{0}(N)$, and fix a representative $Q_{\mathcal{C}}\in\mathcal{C}$. Then $Q_{\mathcal{C}}\in\mathcal{Q}_{D}^{0}$. We define the \emph{(narrow) partial Dedekind zeta function}
attached to the narrow ideal class $[\mathfrak{a}_{Q_{\mathcal{C}}}]\in\mathrm{Cl}^{+}(\mathcal O_K)$ by
\begin{equation}\label{eq:def-narrow-partial-zeta}
\zeta_{[\mathfrak{a}_{Q_{\mathcal{C}}}]}(s)
:=\sum\limits_{\substack{\mathfrak{b}\subset\mathcal{O}_{K} \\ [\mathfrak{b}]=[\mathfrak{a}_{Q_{\mathcal{C}}}]~\text{in}~\mathrm{Cl}^{+}(\mathcal{O}_{K})}}
\mathrm{Nm}(\mathfrak b)^{-s},
\qquad \mathrm{Re}(s)>1,
\end{equation}
where the sum ranges over the integral ideals in the narrow class of $\mathfrak{a}_{Q_{\mathcal{C}}}$.

By Lemmas \ref{lem-pairtoideal} and \ref{lem-stabilizer},
\begin{align}
    \sum\limits_{\substack{(u,v)\in\mathbb{Z}^{2}/\Gamma_{0}(N)_{Q_{\mathcal{C}}} \\ Q_{\mathcal{C}}(u,v)>0}}Q_{\mathcal{C}}(u,v)^{-s}
    &=e_{Q_{\mathcal{C}}}\sum\limits_{\substack{(u,v)\in\mathbb{Z}^{2}/\mathrm{SL}_{2}(\mathbb{Z})_{Q_{\mathcal{C}}} \\ Q_{\mathcal{C}}(u,v)>0}}Q_{\mathcal{C}}(u,v)^{-s}\nonumber\\
    &=e_{Q_{\mathcal{C}}}\sum\limits_{\substack{\mathfrak{b}\subset\mathcal{O}_{K} \\ \mathfrak{b}\sim_{+}\mathfrak{a}_{Q_{\mathcal{C}}}}}
    \mathrm{Nm}(\mathfrak{b})^{-s}
    =e_{Q_{\mathcal{C}}}\cdot\zeta_{[\mathfrak{a}_{Q_{\mathcal{C}}}]}(s).\label{eq-classidealzeta}
\end{align}

For a positive divisor $d$ of $N$, set
$$T_{d}:=\left(\begin{smallmatrix}
    1 & 0 \\ 0 & d
\end{smallmatrix}\right),\quad\Gamma_{0}^{0}(N/d,d):=\left\{\left(\begin{smallmatrix}
    \alpha & \beta \\ \gamma & \delta
\end{smallmatrix}\right)\in\Gamma_{0}(N/d):\beta\equiv 0\pmod{d}\right\}.$$
Define
$$Q_{\mathcal{C}}^{(d)}:=\left[\frac{N}{d}a,b,cd\right]\in\mathcal{Q}_{N/d,D,1}.$$
Since $\Gamma_{0}(N/d)$ acts on $\mathcal{Q}_{N/d,D,1}$ and $\Gamma_{0}^{0}(N/d,d)$ is a subgroup of $\Gamma_{0}(N/d)$, we get an action of $\Gamma_{0}^{0}(N/d,d)$ by restriction, and so the stabilizer $\Gamma_{0}^{0}(N/d,d)_{Q_{\mathcal{C}}^{(d)}}$ makes sense. Moreover, the conjugation by $T_{d}$ identifies this stabilizer with $\Gamma_{0}(N)_{Q_{\mathcal{C}}}$:
$$T_{d}\Gamma_{0}^{0}(N/d,d)_{Q_{\mathcal{C}}^{(d)}}T_{d}^{-1}=\Gamma_{0}(N)_{Q_{\mathcal{C}}}.$$

Note that
$$Q_{\mathcal{C}}(u,dv)=d\cdot Q_{\mathcal{C}}^{(d)}(u,v)$$
for all integers $u,v$. This, together with Lemma \ref{lem}, shows that
\begin{align*}
    \sum\limits_{\substack{(u,v)\in\mathbb{Z}^{2}/\Gamma_{0}(N)_{Q_{\mathcal{C}}} \\ Q_{\mathcal{C}}(u,v)>0}}&Q_{\mathcal{C}}(u,v)^{-s}\\
    &=\sum_{d\mid N}\sum\limits_{\substack{(u,v)\in\mathbb{Z}^{2}/\Gamma_{0}(N)_{Q_{\mathcal{C}}} \\ Q_{\mathcal{C}}(u,v)>0 \\ (v,N)=d}}Q_{\mathcal{C}}(u,v)^{-s}\\
    &=\sum_{d\mid N}\sum\limits_{\substack{(u,v)\in\mathbb{Z}^{2}/\Gamma_{0}^{0}(N/d,d)_{Q_{\mathcal{C}}^{(d)}} \\ Q_{\mathcal{C}}^{(d)}(u,v)>0 \\ (v,N/d)=1}}Q_{\mathcal{C}}(u,dv)^{-s}\\
    &=\sum_{d\mid N}d^{-s}\sum\limits_{\substack{(u,v)\in\mathbb{Z}^{2}/\Gamma_{0}^{0}(N/d,d)_{Q_{\mathcal{C}}^{(d)}} \\ Q_{\mathcal{C}}^{(d)}(u,v)>0 \\ (v,N/d)=1}}Q_{\mathcal{C}}^{(d)}(u,v)^{-s}\\
    &=\sum_{d\mid N}\frac{d^{-s}}{[\Gamma_{0}(N/d)_{Q_{\mathcal{C}}^{(d)}}:\Gamma_{0}^{0}(N/d,d)_{Q_{\mathcal{C}}^{(d)}}]}\sum\limits_{\substack{(u,v)\in\mathbb{Z}^{2}/\Gamma_{0}(N/d)_{Q_{\mathcal{C}}^{(d)}} \\ Q_{\mathcal{C}}^{(d)}(u,v)>0 \\ (v,N/d)=1}}Q_{\mathcal{C}}^{(d)}(u,v)^{-s}.
\end{align*}
Denoting by $\mathcal{C}^{(d)}$ the $\Gamma_{0}(N/d)$-orbit of $Q_{\mathcal{C}}^{(d)}$, we obtain
\begin{equation}\label{eq-zetacomb}
    \sum\limits_{\substack{(u,v)\in\mathbb{Z}^{2}/\Gamma_{0}(N)_{Q_{\mathcal{C}}} \\ Q_{\mathcal{C}}(u,v)>0}}Q_{\mathcal{C}}(u,v)^{-s}
    =\sum_{d\mid N}\frac{d^{-s}}{[\Gamma_{0}(N/d)_{Q_{\mathcal{C}}^{(d)}}:\Gamma_{0}^{0}(N/d,d)_{Q_{\mathcal{C}}^{(d)}}]}\zeta_{\mathcal{C}^{(d)}}^{(N/d)}(s).
\end{equation}

By Corollary \ref{cor-indexone}, \eqref{eq-zetacomb} becomes
\begin{equation}\label{eq-indexremoved}
    \sum\limits_{\substack{(u,v)\in\mathbb{Z}^{2}/\Gamma_{0}(N)_{Q_{\mathcal{C}}} \\ Q_{\mathcal{C}}(u,v)>0}}Q_{\mathcal{C}}(u,v)^{-s}
    =\sum_{d\mid N}d^{-s}\zeta_{\mathcal{C}^{(d)}}^{(N/d)}(s).
\end{equation}

Again, let $\mathcal{C}\in\mathcal{Q}_{N,D,1}/\Gamma_{0}(N)$ and fix $Q_{\mathcal{C}}=[Na,b,c]\in\mathcal{C}$.
For $d\mid N$, set
\[
\mathfrak{d}_{d}:=(d,\omega)\subset\mathcal{O}_{K},
\]
where
$$\omega:=\frac{-1+\sqrt{D}}{2}\in\mathcal{O}_{K}.$$
Since $b\equiv 1\pmod{2N}$ and $d\mid N$, we have $b\equiv 1\pmod{2d}$. Thus,
$$\frac{-b+\sqrt{D}}{2}-\omega=\frac{1-b}{2}\in d\mathbb{Z},$$
and hence
\begin{equation}\label{eq-omegacong}
(d,\tfrac{-b+\sqrt{D}}{2})=(d,\omega)=\mathfrak{d}_{d}.
\end{equation}
Recall that
\[
Q_{\mathcal{C}}^{(d)}:=\left[\frac{N}{d}a,b,cd\right]\in\mathcal{Q}_{N/d,D,1},
\qquad
\mathcal{C}^{(d)}:=Q_{\mathcal{C}}^{(d)}\cdot\Gamma_{0}(N/d)\in \mathcal{Q}_{N/d,D,1}/\Gamma_{0}(N/d).
\]
Since $b\equiv 1\pmod{2N}$, we have $b\equiv 1\pmod{d}$ and hence $\gcd(b,d)=1$.
Applying Lemma \ref{lem-idealtranslation} with $A=Na$ and $B=b$, and using \eqref{eq-omegacong}, we obtain
\begin{equation}\label{eq-classtranslatefixed}
[\mathfrak{a}_{Q_{\mathcal{C}}^{(d)}}]=[\mathfrak{a}_{Q_{\mathcal{C}}}]\cdot [\mathfrak{d}_{d}]^{-1}
\quad\text{in }\mathrm{Cl}^{+}(\mathcal{O}_{K}).
\end{equation}
In particular, the assignment $\mathcal{C}\mapsto \mathcal{C}^{(d)}$ is a bijection
\begin{equation}\label{eq-CtoCdbijection}
\mathcal{Q}_{N,D,1}/\Gamma_{0}(N)\ \xrightarrow{\sim}\ \mathcal{Q}_{N/d,D,1}/\Gamma_{0}(N/d),
\qquad \mathcal{C}\longmapsto \mathcal{C}^{(d)},
\end{equation}
since it corresponds, under $\Theta$, to multiplication by the fixed class $[\mathfrak{d}_{d}]^{-1}$.

Summing \eqref{eq-indexremoved} over $\mathcal{C}\in\mathcal{Q}_{N,D,1}/\Gamma_{0}(N)$ gives
\begin{align}
\sum_{\mathcal{C}\in\mathcal{Q}_{N,D,1}/\Gamma_{0}(N)}
\sum_{\substack{(u,v)\in\mathbb{Z}^{2}/\Gamma_{0}(N)_{Q_{\mathcal{C}}} \\ Q_{\mathcal{C}}(u,v)>0}}
Q_{\mathcal{C}}(u,v)^{-s}
&=\sum_{\mathcal{C}}\sum_{d\mid N}d^{-s}\,\zeta_{\mathcal{C}^{(d)}}^{(N/d)}(s)\nonumber\\
&=\sum_{d\mid N}d^{-s}\sum_{\mathcal{C}}\zeta_{\mathcal{C}^{(d)}}^{(N/d)}(s).\label{eq-swapsums}
\end{align}
Finally, by the bijection \eqref{eq-CtoCdbijection}, we may re-index the inner sum:
\[
\sum_{\mathcal{C}\in\mathcal{Q}_{N,D,1}/\Gamma_{0}(N)}\zeta_{\mathcal{C}^{(d)}}^{(N/d)}(s)
=
\sum_{\mathcal{D}\in\mathcal{Q}_{N/d,D,1}/\Gamma_{0}(N/d)}\zeta_{\mathcal{D}}^{(N/d)}(s)
=:\zeta_{N/d,D,1}(s).
\]
Substituting this into \eqref{eq-swapsums}, we obtain the explicit linear combination
\begin{equation}\label{eq-finallinearcomb}
\sum_{\mathcal{C}\in\mathcal{Q}_{N,D,1}/\Gamma_{0}(N)}
\sum_{\substack{(u,v)\in\mathbb{Z}^{2}/\Gamma_{0}(N)_{Q_{\mathcal{C}}} \\ Q_{\mathcal{C}}(u,v)>0}}
Q_{\mathcal{C}}(u,v)^{-s}
=
\sum_{d\mid N} d^{-s}\,\zeta_{N/d,D,1}(s).
\end{equation}

Recall from \eqref{eq-classidealzeta} that
$$e_{Q_{\mathcal{C}}}\zeta_{[\mathfrak{a}_{Q_{\mathcal{C}}}]}(s)=\sum\limits_{\substack{(u,v)\in\mathbb{Z}^{2}/\Gamma_{0}(N)_{Q_{\mathcal{C}}} \\ Q_{\mathcal{C}}(u,v)>0}}Q_{\mathcal{C}}(u,v)^{-s},$$
where $e_{Q_{\mathcal{C}}}=[\mathrm{SL}_{2}(\mathbb{Z})_{Q_{\mathcal{C}}}:\mathrm{SL}_{2}(\mathbb{Z})_{Q_{\mathcal{C}}}\cap\Gamma_{0}(N)]=[\mathrm{SL}_{2}(\mathbb{Z})_{Q_{\mathcal{C}}}:\Gamma_{0}(N)_{Q_{\mathcal{C}}}]$. By Lemma \ref{lem-stabform}, $e_{Q_{\mathcal{C}}}=1$. It follows that
$$\zeta_{[\mathfrak{a}_{Q_{\mathcal{C}}}]}(s)=\sum\limits_{\substack{(u,v)\in\mathbb{Z}^{2}/\Gamma_{0}(N)_{Q_{\mathcal{C}}} \\ Q_{\mathcal{C}}(u,v)>0}}Q_{\mathcal{C}}(u,v)^{-s}.$$
By Lemma \ref{lem-pairtoideal} and \eqref{eq-finallinearcomb}, we obtain
\begin{align*}
    \zeta_{K}(s)
    &=\sum_{[\mathfrak{a}]\in\mathrm{Cl}^{+}(\mathcal{O}_{K})}\zeta_{[\mathfrak{a}]}(s)
    =\sum_{\mathcal{C}\in\mathcal{Q}_{N,D,1}/\Gamma_{0}(N)}\zeta_{[\mathfrak{a}_{Q_{\mathcal{C}}}]}(s)\\
    &=\sum_{\mathcal{C}\in\mathcal{Q}_{N,D,1}/\Gamma_{0}(N)}\sum_{\substack{(u,v)\in\mathbb{Z}^{2}/\Gamma_{0}(N)_{Q_{\mathcal{C}}} \\ Q_{\mathcal{C}}(u,v)>0}}Q_{\mathcal{C}}(u,v)^{-s}
    =\sum_{d\mid N} d^{-s}\,\zeta_{N/d,D,1}(s).
\end{align*}
Similarly,
\[
    \zeta_{K}(s)=\sum_{d\mid N} d^{-s}\,\zeta_{N/d,D,-1}(s).
\]
Hence,
\begin{equation}\label{eq-dedekindformzetasum}
    \zeta_{K}(k)=\frac{1}{2}\sum_{d\mid N}d^{-k}(\zeta_{N/d,D,1}(k)+\zeta_{N/d,D,-1}(k)).
\end{equation}

Next we fix $1\leq d\mid N$ and analyze the quantity $\zeta_{N/d,D,1}(k)+\zeta_{N/d,D,-1}(k)$. By Lemma \ref{lem-frickeinv}, the form $f_{k,N/d,D,1}^{+}$ lies in $S_{2k}^{+}(\Gamma_{0}(N/d))$. However, by assumption, $S_{2k}^{+}(\Gamma_{0}(N/d))=\{0\}$. Thus, $f_{k,N/d,D,1}^{+}=0$, and hence $r_{f_{k,N/d,D,1}^{+}}(I)(X)=0$. In particular, the leading coefficient of $r_{f_{k,N/d,D,1}^{+}}$ vanishes. As mentioned in the proof of Theorem \ref{thm-zetaformula}, the even-degree terms of $r_{f_{k,N/d,D,1}^{+}}^{+}(I)(X)+r_{f_{k,N/d,D,1}^{-}}^{-}(I)(X)$ coincide with those of $r_{f_{k,N/d,D,1}}(I)(X)$. Using Theorem \ref{thm-period} and the assumption that $k$ is even, we have
\begin{equation*}
    \begin{split}
        0=\sum\limits_{\substack{[Na/d,b,c]\in\mathcal{Q}_{N/d,D,1} \\ a>0>c}}\left(\frac{Na}{d}\right)^{k-1}-&\sum\limits_{\substack{[Na/d,b,c]\in\mathcal{Q}_{N/d,D,1} \\ a<0<c}}\left(\frac{Na}{d}\right)^{k-1}\\
        &-\frac{d\pi}{N(2k-1)c_{k,D}}\cdot\frac{\zeta_{N/d,D,1}(k)+\zeta_{N/d,D,-1}(k)}{\zeta(2k)\prod_{p\mid (N/d)}(1-p^{-2k})}.
    \end{split}
\end{equation*}
It follows that
\begin{align*}
    \zeta_{N/d,D,1}(k)+\zeta_{N/d,D,-1}(k)
    &=\frac{N(2k-1)c_{k,D}\zeta(2k)}{d\pi}\prod_{p\mid (N/d)}(1-p^{-2k})\\
    &\quad\times\left(\sum\limits_{\substack{[Na/d,b,c]\in\mathcal{Q}_{N/d,D,1} \\ a>0>c}}\left(\frac{Na}{d}\right)^{k-1}-\sum\limits_{\substack{[Na/d,b,c]\in\mathcal{Q}_{N/d,D,1} \\ a<0<c}}\left(\frac{Na}{d}\right)^{k-1}\right)\\
    &=d^{-k}N^{k}(2k-1)\binom{2k-2}{k-1}D^{1/2-k}\zeta(2k)\prod_{p\mid(N/d)}(1-p^{-2k})\\
    &\qquad\qquad\qquad\times\left(\sum\limits_{\substack{[Na/d,b,c]\in\mathcal{Q}_{N/d,D,1} \\ a>0>c}}a^{k-1}-\sum\limits_{\substack{[Na/d,b,c]\in\mathcal{Q}_{N/d,D,1} \\ a<0<c}}a^{k-1}\right).
\end{align*}
Inserting this into \eqref{eq-dedekindformzetasum}, we obtain
\begin{align*}
    \zeta_{K}(k)
    &=\frac{1}{2}\sum_{d\mid N}d^{-k}\cdot d^{-k}N^{k}(2k-1)\binom{2k-2}{k-1}D^{1/2-k}\zeta(2k)\prod_{p\mid(N/d)}(1-p^{-2k})\\
    &\qquad\qquad\qquad\qquad\qquad\qquad\times\left(\sum\limits_{\substack{[Na/d,b,c]\in\mathcal{Q}_{N/d,D,1} \\ a>0>c}}a^{k-1}-\sum\limits_{\substack{[Na/d,b,c]\in\mathcal{Q}_{N/d,D,1} \\ a<0<c}}a^{k-1}\right)\\
    &=\frac{1}{2}(2k-1)D^{1/2-k}\zeta(2k)N^{k}\binom{2k-2}{k-1}\\
    &\quad\times\sum_{d\mid N}d^{-2k}\prod_{p\mid(N/d)}(1-p^{-2k})\left(\sum\limits_{\substack{[Na/d,b,c]\in\mathcal{Q}_{N/d,D,1} \\ a>0>c}}a^{k-1}-\sum\limits_{\substack{[Na/d,b,c]\in\mathcal{Q}_{N/d,D,1} \\ a<0<c}}a^{k-1}\right).
\end{align*}

It remains to rewrite the inner sums over $\mathcal{Q}_{N/d,D,1}$ for each positive divisor $d$ of $N$. To this end, fix $d\mid N$ and define the following three sets:
\begin{align*}
    \mathcal{Q}&:=\{[Na/d,b,c]\in\mathcal{Q}_{N/d,D,1}:a>0>c\},\\
    \mathcal{T}&:=\{(a,b,c)\in\mathbb{Z}^{3}:b^{2}-4Nac/d=D,~b\equiv 1\pmod{2N/d},~a>0>c\},\\
    \mathcal{P}&:=\left\{(b,a)\in\mathbb{Z}\times\mathbb{Z}_{>0}:|b|<\sqrt{D},~b\equiv 1\pmod{2N/d},~a~\bigg|~\frac{d(D-b^{2})}{4N}\right\}.
\end{align*}
Then the map $\mathcal{Q}\rightarrow\mathcal{T}$, $[Na/d,b,c]\mapsto(a,b,c)$ is a bijection. Furthermore, the map
$\mathcal{T}\rightarrow\mathcal{P}$, $(a,b,c)\mapsto (b,a)$ is also a bijection. Consequently,
\[
\sum\limits_{\substack{[Na/d,b,c]\in\mathcal{Q}_{N/d,D,1} \\ a>0>c}}a^{k-1}
=\sum\limits_{\substack{|b|<\sqrt{D} \\ b\equiv 1\pmod{2N/d}}}\sum\limits_{\substack{a \big| \frac{d(D-b^{2})}{4N} \\ a>0}}a^{k-1}
=\sum\limits_{\substack{|b|<\sqrt{D} \\ b\equiv 1\pmod{2N/d}}}\sigma_{k-1}\left(\frac{d(D-b^{2})}{4N}\right).
\]
A similar argument shows that
\begin{align*}
    \sum\limits_{\substack{[Na/d,b,c]\in\mathcal{Q}_{N/d,D,1} \\ a<0<c}}a^{k-1}
    &=\sum\limits_{\substack{|b|<\sqrt{D} \\ b\equiv 1\pmod{2N/d}}}\sum\limits_{\substack{a \big| \frac{d(D-b^{2})}{4N} \\ a<0}}a^{k-1}\\
    &=(-1)^{k-1}\sum\limits_{\substack{|b|<\sqrt{D} \\ b\equiv 1\pmod{2N/d}}}\sum\limits_{\substack{a \big| \frac{d(D-b^{2})}{4N} \\ a>0}}a^{k-1}\\
    &=-\sum\limits_{\substack{|b|<\sqrt{D} \\ b\equiv 1\pmod{2N/d}}}\sigma_{k-1}\left(\frac{d(D-b^{2})}{4N}\right).
\end{align*}
Therefore,
\begin{align*}
    \zeta_{K}(k)
    &=(2k-1)D^{1/2-k}\zeta(2k)N^{k}\binom{2k-2}{k-1}\\
    &\quad\times\sum_{d\mid N}d^{-2k}\prod_{p\mid(N/d)}(1-p^{-2k})
    \sum\limits_{\substack{|b|<\sqrt{D} \\ b\equiv 1\pmod{2N/d}}}\sigma_{k-1}\left(\frac{d(D-b^{2})}{4N}\right).
\end{align*}
This completes the proof.
\end{proof}


\begin{thebibliography}{99}

\bibitem{AtkinLehner}
A.~O.~L.~Atkin and J.~Lehner,
\newblock \emph{Hecke operators on $\Gamma_0(m)$},
\newblock \emph{Mathematische Annalen} \textbf{185} (1970), 134--160.
\bibitem{B89} D. A. Buell, \textit{Binary quadratic forms}, Springer, New York, 1989.

\bibitem{CK13} S. Choi and C. H. Kim, \textit{Basis for the space of weakly holomorphic modular forms in higher level cases}, J. Number Theory {\bf 133} (2013), no. 4, 1300--1311.

\bibitem{FT93} A. Fr\"{o}hlich and M. J. Taylor, \textit{Algebraic number theory}, Cambridge Studies in Advanced Mathematics, 27, Cambridge Univ. Press, Cambridge, 1993.

\bibitem{GKZ87} B. H. Gross, W. Kohnen, and D. B. Zagier, \textit{Heegner points and derivatives of $L$-series. II}, Math. Ann. {\bf 278} (1987), no. 1-4, 497--562.

\bibitem{Hecke}
E.~Hecke,
\newblock \emph{Vorlesungen \"uber die Theorie der algebraischen Zahlen},
\newblock Akademische Verlagsgesellschaft, Leipzig, 1923.

\bibitem{K92} A. W. Knapp, \textit{Elliptic curves}, Mathematical Notes, 40, Princeton Univ. Press, Princeton, NJ, 1992.

\bibitem{KZ84} W. Kohnen and D. B. Zagier, \textit{Modular forms with rational periods}, in: Modular forms (Durham, 1983), 197--249, Ellis Horwood Ser. Math. Appl.: Statist. Oper. Res., Horwood, Chichester, 1984.

\bibitem{ManinPeriods}
Y.~I.~Manin,
\newblock \emph{Periods of parabolic forms and $p$-adic Hecke series},
\newblock in: \emph{Arithmetic of Complex Manifolds (Erlangen, 1988)},
World Scientific, 1989.

\bibitem{PP13} V. Pa\c{s}ol and A. A. Popa, \textit{Modular forms and period polynomials}, Proc. Lond. Math. Soc. (3) {\bf 107} (2013), no. 4, 713--743.

\bibitem{Z75} D. B. Zagier, \textit{Modular forms associated to real quadratic fields}, Invent. Math. {\bf 30} (1975), no. 1, 1--46. 
\bibitem{ZagierZetaQuadratic}
D.~Zagier,
\newblock \emph{Zetafunktionen und quadratische K\"orper},
\newblock in: \emph{Zetafunktionen und quadratische K\"orper} (Hochschultext),
Springer, Berlin--Heidelberg, 1981.
\end{thebibliography}
\end{document}